\documentclass[11pt]{amsart}

\usepackage{amsmath,amssymb,mathtools,mathrsfs}
\usepackage[margin=1.12in]{geometry}
\usepackage{enumitem,tikz-cd}
\usepackage{needspace}
\usepackage[colorlinks=true,linkcolor=blue,citecolor=blue,urlcolor=blue]{hyperref}
\usepackage{cleveref}

\setlist[enumerate]{label=\textup{(\roman*)},leftmargin=2.1em}

\numberwithin{equation}{section}
\newtheorem{theorem}{Theorem}[section]
\newtheorem{proposition}[theorem]{Proposition}
\newtheorem{lemma}[theorem]{Lemma}
\newtheorem{corollary}[theorem]{Corollary}
\newtheorem*{maintheorem}{Main Theorem}

\theoremstyle{definition}
\newtheorem{definition}[theorem]{Definition}
\newtheorem{example}[theorem]{Example}
\newtheorem{setup}[theorem]{Setup}
\crefname{setup}{Setup}{Setups}

\theoremstyle{remark}
\newtheorem{remark}[theorem]{Remark}
\newtheorem{claim}[theorem]{Claim}

\DeclareMathOperator{\Spec}{Spec}
\DeclareMathOperator{\Op}{Op}

\DeclareMathOperator{\Pic}{Pic}

\DeclareMathOperator{\Aut}{Aut}

\DeclareMathOperator{\Map}{Map}
\DeclareMathOperator{\Grad}{Grad}
\DeclareMathOperator{\Filt}{Filt}
\DeclareMathOperator{\gr}{gr}
\DeclareMathOperator{\ev}{ev}
\DeclareMathOperator{\id}{id}
\DeclareMathOperator{\rk}{rk}
\DeclareMathOperator{\wt}{wt}
\DeclareMathOperator{\codim}{codim}
\DeclareMathOperator{\ddiv}{div}
\DeclareMathOperator{\coker}{coker}
\DeclareMathOperator{\tors}{tors}
\DeclareMathOperator{\Supp}{Supp}
\DeclareMathOperator{\GL}{GL}
\DeclareMathOperator{\SL}{SL}
\DeclareMathOperator{\PGL}{PGL}
\DeclareMathOperator{\Mat}{Mat}
\DeclareMathOperator{\diag}{diag}
\newcommand{\Gm}{{\mathbb{G}_m}}
\newcommand{\bbA}{\mathbb A}
\newcommand{\PP}{\mathbb P}
\newcommand{\cA}{\mathscr A}

\newcommand{\cO}{\mathcal O}
\newcommand{\cX}{\mathscr X}
\newcommand{\cY}{\mathscr Y}
\newcommand{\cU}{\mathscr U}
\newcommand{\cV}{\mathscr V}
\newcommand{\cD}{\mathscr D}
\newcommand{\cI}{\mathcal I}
\newcommand{\cL}{\mathcal L}
\newcommand{\cN}{\mathcal N}
\newcommand{\cP}{\mathcal P}
\newcommand{\cB}{\mathcal B}
\newcommand{\cK}{\mathcal K}
\newcommand{\cQ}{\mathcal Q}

\newcommand{\rF}{\mathrm F}

\newcommand{\cZ}{\mathscr Z}
\newcommand{\sst}{\mathrm{ss}}
\newcommand{\red}{\mathrm{red}}

\newcommand{\GIT}{\mathrm{GIT}}
\newcommand{\PD}{\mathrm{PD}}

\newcommand{\bbC}{\mathbb C}
\newcommand{\bbQ}{\mathbb Q}
\newcommand{\bbZ}{\mathbb Z}
\newcommand{\bbF}{\mathbb F}
\newcommand{\bbT}{\mathbb T}
\newcommand{\rmT}{\mathrm T}

\title[Open substacks preserving moduli spaces: classification and bounds]{Open substacks preserving moduli spaces: \\ classification and bounds}
\author{Jianping Wang and Xucheng Zhang}
\address{Beijing Institute of Mathematical Sciences and Applications, Beijing 101408, China}
\address{Email: jianpw@bimsa.cn}
\address{School of Mathematical Sciences, Beihang University, Beijing 102206, China}
\address{Email: xuchengz@buaa.edu.cn}

\date{}

\begin{document}
\linespread{1.15}\selectfont
\begin{abstract}
For a smooth locally reductive stack admitting an adequate moduli space, we show that there are finitely many open substacks preserving the adequate moduli space, and give a complete classification of such opens. For $\Gm$-actions, this becomes an explicit description in terms of the codimension-one Bia\l{}ynicki-Birula strata. Moreover, for a stable adequate moduli space, we show that the number of preserving opens are bounded by $3$ raised to the relative Picard rank. As an application, the number of fixed components for $\Gm$-actions is at least the relative Picard rank.
\end{abstract}

\maketitle

\section{Introduction}
Passing to an open substack can change the moduli spaces substantially, as shown by the typical example $[\bbA^2-\{0\}/\Gm] \subseteq [\bbA^2/\Gm]$, where $\Gm$ acts on $\bbA^2$ with weight $(1,1)$. In particular,  birationality does not descend to moduli spaces. In this paper, we study the extreme case in which an open substack shares the same moduli space with the ambient stack.

Let us make this precise. For any adequate moduli space $\pi: \cX \to X$, an open substack $\cU \subseteq \cX$ is said to \emph{preserve} $\pi$ if the restriction $\pi|_{\cU}: \cU \to X$ is still an adequate moduli space. One can artificially construct many such opens in the case of quotient stacks. For example, let $\Gm$ act on $\bbA^2$ with weight $(+1,-1)$, removing either coordinate axis preserves the good quotient (while removing their union does not). More natural examples occur in moduli theory, and we only mention few of them here. It is quite often that chamber-to-wall open immersions in K-moduli induce isomorphisms on the good moduli spaces, examples include: the first wall for plane curves \cite[Thm. 5.5]{ADL}, the second wall for plane quintic curves \cite[Thm. 7.1 (2)]{ADL}, and the wall at $c=1/3$ for quartic K3 surfaces \cite[Thm. 5.14 (1)-(2)]{ADLQuarticK3}. There are also examples from the Hassett-Keel program: the wall at $\alpha=9/11$ for stable pointed curves \cite[Main Thm. (of the trilogy)]{second-flip}. Another example is the inclusion of the GIT-stack of semistable binary into the stack of boundary-polarized Calabi-Yau curves \cite[Thm. 16.9 (2)]{ABBDeVILW23}. In general, extended weighted blow-ups provide a systematic construction of enlargements preserving the good moduli spaces \cite[Thm. 1.7]{DLIStableMaps}.

These examples lead to natural questions: which open substacks preserve a given adequate moduli space, and whether there can be infinitely many of them. Our main result answers this completely for smooth locally reductive stacks.
\begin{maintheorem}[\Cref{thm:geometric-good} and \Cref{cor:good-finite}]
Let $k$ be an algebraically closed field and let $\cX$ be a smooth locally reductive algebraic stack of finite type over $k$, with separated diagonal. Suppose that $\cX$ admits an adequate moduli
space $\pi: \cX \to X$. Then only finitely many open substacks of $\cX$ preserve $\pi$. Their reduced boundaries are unions of prime divisors taken from a fixed finite collection, and the collections of divisors that can be removed simultaneously form a finite simplicial complex determined by boundary weights.
\end{maintheorem}
\begin{remark}
One of the main difficulties in proving \textbf{Main Theorem} lies in identifying the correct characterization of the boundary. This is achieved by verifying many artificially constructed quotient stacks, especially those arising from torus actions. Once this is done, the generalization to locally reductive stacks is fairly straightforward. A distinctive feature of the proof is its combination of classical and modern methods: GIT handles the local case (quotient stacks), while stack-theoretic machinery handles the global case (locally reductive stacks).
\end{remark}
The proof begins with a criterion valid for arbitrary adequate moduli space. The open immersion from a preserving open $\cU \subseteq \cX$ must be affine, so its reduced boundary $(\cX \setminus \cU)_\red$ is an effective Cartier divisor (\Cref{lem:boundary-purity}), and we call such a boundary \emph{removable}. For a reduced effective Cartier divisor $i: \cD \hookrightarrow \cX$, let $\cN=\cN_{\cD/\cX}$ be its normal line bundle, and we prove that (\Cref{thm:adequate-boundary})
\[
\cX \setminus \cD \text{ preserves } \pi \Leftrightarrow (\pi \circ i)_*\cN^{\otimes n} = 0 \text{ for every } n > 0.
\]
The key point is that adequate affineness makes positive-degree sections of the normal line bundle nilpotent, and reducedness of $\cD$ forces them to vanish. This criterion also gives Picard-theoretic obstructions to remove a boundary (\Cref{prop:descent-obstruction}).

We next turn to a geometric characterization of the removable boundary. For quotient stacks arising from reductive group actions, the key input is the Hilbert–Mumford criterion from Geometric Invariant Theory (GIT). To facilitate its adaptation to general algebraic stacks, we use a reformulation of the Hilbert–Mumford criterion in terms of filtrations (\cite[\S 1.D]{Heinloth})
\[
f: \Theta:=[\bbA^1/\Gm] \to \cX.
\]

Using this notation, a reduced effective Cartier divisor $\cD$ on $\cX$ is removable precisely when every geometric point of $\cD$ is realized as $f(1)$ for some filtration $f: \Theta \to \cX$ such that the weight of $f^*\cO_{\cX}(\cD)$ at $0:=\mathrm{B}\mathbb{G}_m \in \Theta$ is negative (see condition (\S) in the proof of \Cref{thm:quotient-boundary}). Since the discrete weight is constant on connected components, this condition forces any removable boundary $\cD$ to contain the $G$-saturations of the components of the attracting loci on which the line bundle $\cO_{\cX}(\cD)$ has negative weight (\Cref{cor:weight-inclusion}). In particular, the prime removable boundaries occur among the individual $G$-saturations, and a union of prime removable boundaries is removable if the weight condition is satisfied (\Cref{cor:finite-reductive}). Finiteness already follows because only finitely many distinct attracting loci occur (\Cref{lem:quotient-images}), even though there may be infinitely many cocharacters. 

The passage from quotient stacks to locally reductive stacks is straightforward, once we observe that preserving of opens can be checked \'{e}tale locally on the moduli space (\Cref{prop:chart-test}), and the very definition of local reductivity supplies precisely the required local presentations. The filtration reformulation of Hilbert–Mumford criterion above then concisely expresses the removable prime boundary $\cD$ as the evaluation image in $\cX$ of a negative-weight connected component of the filtration stack $\Filt(\cX)$ (Corollary \ref{cor:good-finite}).

For a \emph{stable} adequate moduli space (in the sense of \Cref{defn:stable_adequate_moduli}), we can give a simple and sharp bound for the numbers of preserving opens via the relative Picard rank (\Cref{thm:stack-picard-bound}). This relative Picard rank is finite under very mild assumptions (\Cref{prop:relative-picard-finiteness}), and for quotient stacks there are some estimates for it using ordinary Picard group (\Cref{cor:reductive-relative-picard-comparison}). In particular, if $\Pic(X)$ is torsion, the number of preserving opens is at most $3^{\rk G/[G,G]}$ (\Cref{cor:ordinary-picard-counting}). Examples from torus actions show that the ``stable'' hypothesis cannot be dropped in general (\Cref{ex:torus-no-stable-unbounded}). As an application, we show that for $\Gm$-actions, the number of the fixed components is at least the relative Picard rank (\Cref{prop:fixed-weight-kernel}). In particular, if $\Pic(X)$ is torsion, there are at most 3 preserving opens (\Cref{cor:fixed-torsion-picard}).

\medskip
\noindent\textbf{Organization of the paper.} \Cref{sec:general} develops the general boundary criterion. \Cref{sec:reductive} proves the geometric criterion in the case of quotient stacks, and it is globalized to locally reductive stacks in \Cref{sec:locally-reductive} using Cartesian \'{e}tale quotient charts. Finally, \Cref{sec:bound} bounds the number of preserving opens in terms of the relative Picard rank, and also relates this rank to the number of fixed components for $\Gm$-actions.

\medskip
\noindent\textbf{Conventions.}
All algebraic stacks and algebraic spaces are quasi-separated.
A variety is a separated scheme of finite type over an algebraically closed field. A reductive group scheme is a smooth affine group scheme of finite presentation with connected reductive geometric fibers.

\medskip
\noindent\textbf{AI declarations.} The proof that a removable boundary for quotient stacks must be a $G$-saturation was obtained by the authors. In exploring the converse direction, an iterative discussion with ChatGPT 6 Astra led the authors to identify the negative-weight condition needed to guarantee that such a $G$-saturation is removable; this observation was then independently developed and verified by the authors. ChatGPT 6 Astra was also used to check that the fixed locus in \Cref{cor:affine-space-three-opens} is non-empty and connected, and for proofreading.

\subsection*{Acknowledgment}
This work began when both authors were postdocs at YMSC, Tsinghua University. We thank the institution for its excellent working conditions. X.C. Zhang was partially supported by the Fundamental Research Funds for the Central Universities (Grant No. KG16420801 and KG16492301).

\section{General properties}\label{sec:general}

\begin{definition}
Let $\pi: \cX \to X$ be an adequate moduli space. An open substack $\cU \subseteq \cX$ is said to \emph{preserve $\pi$} if the restriction $\pi|_{\cU}: \cU \to X$ is an adequate moduli space. Equivalently, the open substack $\cU$ admits an adequate moduli space $\cU \to U$ and the induced morphism between adequate moduli spaces
\[
\begin{tikzcd}
\cU \ar[r,hook] \ar[d] & \cX \ar[d,"\pi"] \\
U \ar[r,"\sim"] & X
\end{tikzcd}
\]
is an isomorphism. Denote by $\Op(\pi)$ the set of open substacks of $\cX$ preserving $\pi$, including $\cX$ itself.

A closed substack $\cD \subseteq \cX$ is called \emph{removable} for $\pi$ if its open complement $\cX \setminus \cD \subseteq \cX$ preserves $\pi$.
\end{definition}
\begin{lemma}[Preserving opens through a gerbe]\label{lem:gerbe}
Let $\pi: \cX \to X$ be an adequate moduli space, and let $f: \cY \to \cX$ be a quasi-compact quasi-separated fppf gerbe, banded by a linearly reductive group scheme over $\cX$. Then $\pi \circ f: \cY \to X$ is an adequate moduli space. Moreover, pullback along $f$ induces a bijection between the open substacks preserving $\pi$ and those preserving $\pi \circ f$.
\end{lemma}
\begin{proof}
We first show that $\pi \circ f$ is an adequate moduli space. Fppf locally on $\cX$, the gerbe $f$ is of the form $\cX \times \mathrm{B}G \to \cX$. Linear reductivity of $G$ implies that $f_*$ is exact on quasi-coherent modules and $f_*\cO_{\cY}=\cO_{\cX}$. Thus $f$ is cohomologically affine, and $\pi \circ f$ is adequately affine by
\cite[Prop. 4.2.1 (1)--(2)]{AlperAdequate}. Since $(\pi \circ f)_*\cO_{\cY} \simeq \pi_*\cO_{\cX} \simeq \cO_X$, the composition $\pi \circ f$ is an adequate moduli space.

A gerbe is a universal homeomorphism \cite[\href{https://stacks.math.columbia.edu/tag/06R9}{Tag 06R9}]{stacks-project}. So every open substack $\cU \subseteq \cY$ is of the form $f^{-1}(\cV)$ for some unique open substack $\cV \subseteq \cX$. Let
\[
f_{\cV}: \cU \to \cV \text{ and } \pi_{\cV}:=\pi|_{\cV}.
\]
The above arguments apply to $f_{\cV}$. In particular, if $\cV$ preserves $\pi$, then $\cU$ preserves $\pi \circ f$.

Conversely, suppose that $\pi_{\cV} \circ f_{\cV}: \cU \to X$ is an adequate moduli space. Then $(\pi_{\cV})_*\cO_{\cV} \simeq (\pi_{\cV} \circ f_{\cV})_*\cO_{\cU} \simeq \cO_X$. To check adequate affineness, let $\mathcal{A} \twoheadrightarrow \cB$ be a surjection of quasi-coherent $\cO_{\cV}$-algebras. Using the adequate affineness of $\pi_{\cV} \circ f_{\cV}$, we obtain a
universally adequate homomorphism
\[
(\pi_{\cV} \circ f_{\cV})_* \circ f_{\cV}^*\mathcal{A} \to (\pi_{\cV} \circ f_{\cV})_* \circ f_{\cV}^*\cB.
\]
The adjunction isomorphisms identify this homomorphism with $(\pi_{\cV})_*\mathcal{A} \to (\pi_{\cV})_*\cB$. Hence $\pi_{\cV}$ is adequately affine and therefore
an adequate moduli space.
\end{proof}
First we observe that the inclusion from preserving opens is affine.
\begin{lemma}\label{lem:affine-open}
Let $\cX$ be an algebraic stack with affine diagonal, and let $\pi: \cX \to X$ be an adequate moduli space. If an open immersion $j: \cU \hookrightarrow \cX$ preserves $\pi$, then $j$ is affine.
\end{lemma}
\begin{proof}
Affineness of $j$ can be checked after an \'{e}tale covering of $X$, and adequate moduli spaces are preserved by \'{e}tale base change. We may assume that $X$ is affine. Because $\pi$ is quasi-compact, $\cX$ is quasi-compact, so it admits a smooth surjective atlas $p: Y \to \cX$ with $Y$ affine. To conclude, we show the base change $j_Y: Y \times_\cX \cU \to Y$ is affine. First we show that $p$ is affine. Consider the Cartesian diagram
\[
\begin{tikzcd}
Y \ar[r,"\Gamma_p"] \ar[d] & Y \times_X \cX \ar[r] \ar[d] & Y \ar[d,"p"'] & Y \times_\cX \cU \ar[d,"p_{\cU}"] \ar[l,hook,"j_Y"'] \\
\cX \ar[r,"\Delta_\pi"'] & \cX \times_X \cX \ar[r] \ar[d] & \cX \ar[d,"\pi"] & \cU \ar[l,hook,"j"] \\
& \cX \ar[r,"\pi"'] & X
\end{tikzcd}
\]
So the atlas $p: Y \to \cX$ factors as $Y \xrightarrow{\Gamma_p} Y \times_X \cX \xrightarrow{\mathrm{pr}_2} \cX$. The graph $\Gamma_p$ is a base change of the diagonal $\Delta_\pi$ of $\pi$, which is again a base change of the diagonal of $\cX$, so is affine.  The second projection $\mathrm{pr}_2$ is a base change of $\pi \circ p: Y \to X$, which is affine as both $Y$ and $X$ are affine. Thus $p$ is affine. Then the morphism $p_{\cU}: Y \times_{\cX} \cU \to \cU$ is affine. Note that $\pi|_{\cU}: \cU \to X$ is qcqs. Consequently, the composition $Y \times_{\cX} \cU \xrightarrow{p_{\cU}} \cU
\xrightarrow{\pi|_{\cU}} X$ is adequately affine and qcqs, and hence affine \cite[Thm. 4.3.1]{AlperAdequate}. Since $X$ is affine, $Y \times_{\cX} \cU$ is also affine. Then $j_Y: Y \times_{\cX} \cU \to Y$ is a morphism between affine schemes, hence affine.
\end{proof}
\Cref{lem:affine-open} separates the two parts in the definition of adequate moduli spaces. Once the open immersion is affine, adequate affineness is automatic and only the condition on regular functions needs to be checked. We record this as a corollary.
\begin{corollary}\label{cor:affine-test}
Let $\cX$ be an algebraic stack with affine diagonal, and let $\pi: \cX \to X$ be an adequate moduli space. Then an open immersion $j: \cU \hookrightarrow \cX$ preserves
$\pi$ if and only if $j$ is affine and the canonical map
$\cO_X \to (\pi \circ j)_*\cO_{\cU}$ is an isomorphism.
\end{corollary}
Next we translate affineness into a restriction on the boundary. On an irreducible or normal stack it rules out boundary components of codimension larger than $1$.
\begin{lemma}[Purity of the boundary]\label{lem:boundary-purity}
Let $\cX$ be a locally noetherian algebraic stack with affine diagonal, and let $\pi: \cX \to X$ be an adequate moduli space. Let $j: \cU \hookrightarrow \cX$ be a non-trivial open immersion preserving $\pi$. If $j$ is dense (e.g., if $\cX$ is irreducible or normal), then the boundary $\cX \setminus \cU$ is pure of codimension $1$. If $\cX$ is locally factorial, its reduced boundary $(\cX \setminus \cU)_\red$ is an effective Cartier divisor on $\cX$.
\end{lemma}
\begin{proof}
We first verify the assertions on density. If $\cX$ is irreducible, then $\cU \neq \varnothing$ as $\pi|_{\cU}$ is surjective, and hence is dense. If $\cX$ is normal, its irreducible components are open and closed. Moreover, preservation of $\pi$ gives an isomorphism $\Gamma(\cX,\cO_{\cX}) \xrightarrow{\sim} \Gamma(\cU,\cO_{\cU})$ since both sides are canonically identified with $\Gamma(X,\cO_X)$. If $\cU$ missed an irreducible component $\cX^\circ$ of $\cX$, the non-zero idempotent equal to $1$ on $\cX^\circ$ and $0$ on its complement would restrict to $0$ on $\cU$, a contradiction. Thus $\cU$ meets every irreducible component of $\cX$ and is dense.

The open immersion $j$ is affine by \Cref{lem:affine-open}.  On a locally noetherian atlas, it becomes a dense affine open immersion of schemes, so each irreducible component of the complement has codimension $1$ \cite[\href{https://stacks.math.columbia.edu/tag/0BCU}{Tag 0BCU}]{stacks-project}. Codimension can be checked on an atlas, so the same holds for $\cX \setminus \cU$. If $\cX$ is locally factorial, the Weil divisor $(\cX \setminus \cU)_\red$ is Cartier.
\end{proof}
By \Cref{lem:boundary-purity}, for a locally reductive stack, reduced removable closed substacks are necessarily effective Cartier divisors. It remains to determine when an effective Cartier divisor is removable. By definition, a closed substack is removable if and only if its reduction is. Hence it suffices to determine when a \emph{reduced} effective Cartier divisor is removable. The following theorem characterizes this.
\begin{theorem}\label{thm:adequate-boundary}
Let $\pi: \cX \to X$ be an adequate moduli space. Let $i: \cD \hookrightarrow \cX$ be a reduced effective Cartier divisor on $\cX$ with normal bundle $\cN:=\cN_{\cD/\cX}=\cO_{\cX}(\cD)|_\cD$. Then the following are equivalent:
\begin{enumerate}
\item 
the closed substack $\cD \subseteq \cX$ is removable for $\pi$.
\item 
the canonical map $\cO_X \to \pi_*\cO_{\cX}(n\cD)$ is an isomorphism for every integer $n > 0$.
\item
the condition $(\pi \circ i)_*\cN^{\otimes n}=0$ holds for every integer $n>0$.
\end{enumerate}
\end{theorem}
Note that vanishing of positive powers of the boundary normal bundle also appears in the criterion for extended weighted blow-ups of \cite[Prop. A.1]{DLIStableMaps}.
\begin{proof}
Let $j: \cU:=\cX \setminus \cD \hookrightarrow \cX$ be the inclusion of the open complement of $\cD$. First we claim that $j$ is affine. Indeed, affineness can be checked smooth local on the target $\cX$, so this reduces to the scheme case \cite[\href{https://stacks.math.columbia.edu/tag/07ZU}{Tag 07ZU}]{stacks-project}, as smooth pullback preserves effective Cartier divisors. 

(i) $\Leftrightarrow$ (ii). Since $j$ is affine, the composition $\pi \circ j$ is adequately affine \cite[Prop. 4.2.1 (2) and (1)]{AlperAdequate}. Therefore, the restriction $\pi|_{\cU}=\pi \circ j: \cU \to X$ is an adequate moduli space if and only if the canonical map $\cO_X \to (\pi \circ j)_*\cO_{\cU}$ is an isomorphism. To compute $(\pi \circ j)_*\cO_{\cU}$, we note that
\[
j_* \cO_{\cU} = \varinjlim_{n \geq 0} \cO_{\cX}(n\cD).
\]
Indeed, this can be checked smooth locally on the target $\cX$, on an affine chart $\Spec A$ where $\cD$ is defined by a single non-zero-divisor $f$. On such a chart,
\[
A_f = \varinjlim \left( A \xrightarrow{\cdot f} A \xrightarrow{\cdot f} A \to \cdots \right) = \bigcup_{n \geq 0} f^{-n}A,
\]
as every element of $A_f$ belongs to $f^{-n}A$ for some integer $n \geq 0$. Pushforward along a qcqs morphism commutes with filtered colimits of quasi-coherent modules, so
\[
(\pi \circ j)_*\cO_{\cU} = \pi_* \left(j_* \cO_{\cU} \right) = \pi_*\varinjlim_{n \geq 0} \cO_{\cX}(n\cD) = \varinjlim_{n \geq 0} \pi_*\cO_{\cX}(n\cD) = \bigcup_{n \geq 0} \pi_*\cO_{\cX}(n\cD)
\]
The last equality holds because the transition maps $\pi_*\cO_{\cX}(n\cD) \hookrightarrow \pi_*\cO_{\cX}((n+1)\cD)$ are injective. In particular, every $\pi_*\cO_{\cX}(n\cD)$ embeds into this colimit. For every integer $n \geq 0$, the canonical map $\cO_X \to (\pi \circ j)_*\cO_{\cU}$ factors as
\[
\cO_X \xrightarrow{\sim} \pi_*\cO_{\cX} \hookrightarrow \pi_*\cO_{\cX}(n\cD) \hookrightarrow \bigcup_{n \geq 0} \pi_*\cO_{\cX}(n\cD) = (\pi \circ j)_*\cO_{\cU}.
\]
Consequently, if $\cO_X \to (\pi \circ j)_*\cO_{\cU}$ is an isomorphism, then every $\cO_X \to \pi_*\cO_{\cX}(n\cD)$ is an isomorphism. Conversely, if every canonical map $\cO_X \to \pi_*\cO_{\cX}(n\cD)$ is an isomorphism, their colimit is also $\cO_X$. This proves the equivalence (i) $\Leftrightarrow$ (ii).

(iii) $\Rightarrow$ (ii). For every integer $n>0$ there is a short exact sequence of $\cO_\cX$-modules
\[
0 \to \cO_\cX((n-1)\cD) \xrightarrow{\cdot s} \cO_\cX(n\cD) \to i_*\cN^{\otimes n} \to 0,
\]
where $s$ is the canonical section of $\cO_\cX(\cD)$. Applying $\pi_*$ gives
\[
0 \to \pi_*\cO_\cX((n-1)\cD) \to \pi_*\cO_\cX(n\cD) \to (\pi \circ i)_*\cN^{\otimes n}.
\]
Under (iii), the first map is an isomorphism. Induction, starting from $\cO_X \xrightarrow{\sim} \pi_*\cO_\cX$, proves (ii).

(ii) $\Rightarrow$ (iii). The idea is that adequateness will imply that every local section of
$(\pi \circ i)_*\cN^{\otimes n}$ has a positive tensor power equal to zero. Reducedness will then force the section itself to vanish. Consider the graded quasi-coherent $\cO_{\cX}$-algebras
\[
\cB:=\bigoplus_{n \geq 0} \cO_\cX(n\cD) \text{ and } \cQ:=\cB/(s),
\]
where $s$ is regarded as a degree $1$ element. The degree $0$ part of the ideal $(s)$ is zero and degree $n>0$
part is $s\cO_\cX((n-1)\cD)$. Consequently,
\[
\cQ_0 = \cO_{\cX} \text{ and } \cQ_n = \cO_\cX(n\cD)/s\cO_\cX((n-1)\cD) \text{ for } n>0.
\]
The short exact sequence above gives
$\cQ_n \simeq i_*\cN^{\otimes n}$ for every $n>0$, and these identifications are compatible with multiplication in positive degrees. In particular, the power $h^r$ of a homogeneous section of degree $n>0$ corresponds to its tensor power $h^{\otimes r}$ as a section of $\cN^{\otimes nr}$.

By (ii), together with $\cO_X \xrightarrow{\sim} \pi_*\cO_\cX$ for $n=0$, the canonical maps $\cO_X \to \pi_*\cO_\cX(n\cD)$ sending $a \mapsto (\pi^*a)s^n$ are isomorphisms for all $n \geq 0$. Since $\pi$ is qcqs, its pushforward on quasi-coherent modules commutes with direct sums. This gives an isomorphism $\cO_X[s] \xrightarrow{\sim} \pi_*\cB$ of graded $\cO_X$-algebras. In particular, every degree $n$ local section of $\pi_*\cO_\cX(n\cD)$ is a multiple of $s^n$.

Let us fix an integer $n>0$. For any affine \'{e}tale chart $W \to X$ and section $h$ of $(\pi \circ i)_*\cN^{\otimes n}$ over $W$, under the identification $(\pi \circ i)_*\cN^{\otimes n} \simeq \pi_*\cQ_n$, one may view $h$ as a section of $\pi_*\cQ$ over $W$. Since $\pi$ is adequately affine, for the quotient of quasi-coherent $\cO_\cX$-algebras $\cB \twoheadrightarrow \cQ$ and the section $h$ of $\pi_*\cQ$, after an \'{e}tale refinement of $W$ if necessary, some positive power $h^r$ lifts to $\pi_*\cB$. The quotient map preserves the grading, and $h^r$ is homogeneous of degree $nr$. Taking the degree $nr$ component gives a lift of $h^r$ in $\pi_*\cO_\cX(nr\cD)$. But every such lift is a multiple of $s^{nr}$, whose image in $\cQ$ is zero. This shows that $h^r=0$, hence $h^{\otimes r}=0$.

Since $W \to X$ is \'{e}tale and $\cD$ is reduced, $\cD \times_X W$ is reduced. On a reduced stack, a section of a line bundle whose positive tensor power is zero is itself zero, as is checked by trivializing the line bundle on a reduced atlas. Thus $h=0$ \'{e}tale locally on $W$, and hence on $W$.
\end{proof}

\begin{example}\label{ex:adequate-not-good}
Let $k$ be an algebraically closed field of characteristic $p>0$ and let
$\Gm \times \SL_2$ act on $\bbA^1$ by $(a,h).x=ax$. There is an adequate moduli space $[\bbA^1/(\Gm \times \SL_2)] \to \Spec k$. Removing the zero divisor gives $\mathrm{BSL}_2$ with the same
adequate moduli space. This follows directly from computing the invariant rings, or from \Cref{thm:adequate-boundary}: each positive power
of the normal bundle has a non-zero $\Gm$-character.
\end{example}
\begin{corollary}\label{cor:removable-restriction}
Let $\pi: \cX \to X$ be an adequate moduli space. If $\cD' \subseteq \cD$ are effective Cartier divisors on $\cX$ and $\cD$ is removable for $\pi$, then $\cD'$ is removable for $\pi$.
\end{corollary}
\begin{proof}
For every integer $n > 0$, we have canonical inclusions $\cO_X \hookrightarrow \pi_*\cO_\cX(n\cD') \hookrightarrow \pi_*\cO_\cX(n\cD)$. The composition is an isomorphism by \Cref{thm:adequate-boundary}, so the first map is also an isomorphism. Then applying \Cref{thm:adequate-boundary} to $\cD'$ proves the result.
\end{proof}
The criterion in \Cref{thm:adequate-boundary} also yields a useful
obstruction in terms of descent: if the complement of an effective Cartier divisor $\cD$ on $\cX$ preserves $\pi$, then no positive tensor power of $\cO_\cX(\cD)$ descends to the moduli space. The key point is that the corresponding power of the canonical section would descend as well.
\begin{proposition}[Descent obstruction]\label{prop:descent-obstruction}
Let $\cX$ be an algebraic stack admitting an adequate moduli space $\pi: \cX \to X$. Let $\varnothing \neq \cD \subseteq \cX$ be an effective Cartier divisor. If $\cO_\cX(n\cD) \simeq \pi^*A$ for some $n>0$ and $A \in \Pic(X)$, then $\cD$ is not removable for $\pi$.
\end{proposition}
\begin{proof}
Let $\cU:=\cX\setminus\cD$, and let
$s\in\Gamma(\cX,\cO_\cX(\cD))$ be the canonical section defining
$\cD$. We conclude by showing that $\pi|_{\cU}: \cU \to X$ is not surjective. Fix an isomorphism
$\cO_\cX(n\cD) \simeq \pi^*A$ and use it to regard
$s^{\otimes n}$ as a section of $\pi^*A$. Since $\pi$ is an adequate moduli space, the projection formula for $A$ gives canonical isomorphisms $A \xrightarrow{\sim} A \otimes_{\cO_X} \pi_*\cO_\cX \xrightarrow{\sim} \pi_*\pi^*A$. So there is a unique section $t \in \Gamma(X,A)$
whose pullback is $s^{\otimes n}$. Let $Z(t) \subseteq X$ be its zero locus. Since zero loci of sections of line bundles commute with base change, we have an equality of closed substacks $\cX \times_X Z(t) = Z(\pi^*t) = Z(s^{\otimes n}) = n\cD$. As $|n\cD|=|\cD| \neq \varnothing$, this implies that
$Z(t) \neq \varnothing$. Taking complements gives $\cU=\cX \setminus \cD = \cX \times_X (X \setminus Z(t))$. Thus $\pi|_\cU$ factors through the open subspace
$X \setminus Z(t) \subsetneq X$ and is not surjective.
\end{proof}
Putting \Cref{lem:boundary-purity} and \Cref{thm:adequate-boundary} together, we obtain a bijection between certain Cartier divisors and preserving opens for locally factorial stacks.
\begin{theorem}\label{thm:divisorial-classification}
Let $\cX$ be a locally factorial algebraic stack with affine diagonal, and let $\pi: \cX \to X$ be an adequate moduli space. Taking complements gives a bijection from reduced effective Cartier divisors $\cD$ on $\cX$ satisfying
\[
(\pi \circ i)_*\cO_{\cD}(n\cD)=0 \text{ for every integer } n>0
\]
to $\Op(\pi)$, where $i: \cD \hookrightarrow \cX$ is the closed immersion.
\end{theorem}
\begin{proof}
\Cref{thm:adequate-boundary} shows that taking complements sends every reduced effective Cartier divisor satisfying the vanishing condition to an open substack preserving $\pi$. This assignment is injective because a reduced closed substack is determined by its complement. Conversely, let $\cU \subsetneq \cX$ be an open substack preserving $\pi$. Then $\cD=(\cX \setminus \cU)_{\red}$ is an effective Cartier divisor by \Cref{lem:boundary-purity}. Applying \Cref{thm:adequate-boundary} to $\cD$ gives
$(\pi \circ i)_*\cO_{\cD}(n\cD)=0$ for every $n>0$. This shows the surjectivity.
\end{proof}
The preceding correspondence has a simplicial description: the faces record which prime divisors can be removed simultaneously. For an adequate moduli space $\pi: \cX \to X$, let $\PD(\pi)$ be the set of removable reduced prime divisors on $\cX$.
\begin{corollary}\label{cor:simplicial-correspondence}
Let $\cX$ be a noetherian locally factorial algebraic stack with affine
diagonal, and let $\pi: \cX \to X$ be an adequate moduli space. For a finite subset $\sigma \subseteq \PD(\pi)$, write
\[
\cD_\sigma:=\sum_{\cD \in \sigma} \cD \text{ and } i_\sigma: \cD_\sigma \hookrightarrow \cX.
\]
Let
\[
\cK(\pi):=\left\{ \sigma \subseteq \PD(\pi) \text{ finite}: (\pi \circ i_\sigma)_*\cO_{\cD_\sigma}(n\cD_\sigma)=0 \text{ for every } n>0\right\}.
\]
Then $\cK(\pi)$ is an abstract simplicial complex with vertex set $\PD(\pi)$, and a finite collection of vertices forms a face precisely when the complement of their union preserves $\pi$. The assignment $\sigma \mapsto \cX \setminus \cD_\sigma$ is an order-reversing bijection from its faces to $\Op(\pi)$.
\end{corollary}

\begin{proof}
Local factoriality ensures that every prime divisor on $\cX$ is Cartier
and that a finite sum of distinct prime divisors is reduced.

The empty set belongs to $\cK(\pi)$, and every singleton
$\{\cD\}$ with $\cD \in \PD(\pi)$ belongs to $\cK(\pi)$ by
\Cref{thm:adequate-boundary}. If $\sigma \in \cK(\pi)$ and $\tau \subseteq \sigma$, then $\cU_\sigma$ preserves
$\pi$ and $\cD_\tau \leq \cD_\sigma$. \Cref{cor:removable-restriction} shows that $\cU_\tau$ preserves $\pi$. Since $\cD_\tau$ is reduced, \Cref{thm:adequate-boundary} gives $\tau \in\cK(\pi)$.
Thus $\cK(\pi)$ is an abstract simplicial complex with the vertex set $\PD(\pi)$.

By \Cref{thm:adequate-boundary}, every face $\sigma$ gives an open substack $\cU_\sigma$ preserving $\pi$. Conversely, let $\cU \subseteq \cX$ be an open substack preserving $\pi$. The reduced complement $\cD:=(\cX \setminus \cU)_{\red}$ is an effective Cartier divisor by \Cref{lem:boundary-purity}. Since $\cX$ is noetherian, $\cD$ has finitely many prime
components. Let $\sigma$ be their set. Applying \Cref{cor:removable-restriction} to each prime component of $\cD$
shows that $\sigma \subseteq \PD(\pi)$, and $\cD=\cD_\sigma$. Another application of \Cref{thm:adequate-boundary} gives $\sigma \in \cK(\pi)$, with $\cU=\cU_\sigma$.
\end{proof}
For the moment, it is not immediately clear from \Cref{thm:divisorial-classification} or \Cref{cor:simplicial-correspondence} whether there are finitely many such preserving opens; nevertheless, we will show shortly that when $\cX$ is locally reductive, the faces can be characterized in terms of boundary weights, thereby establishing finiteness of such open subsets.

\section{Quotient stacks}\label{sec:reductive}

\subsection{Recollections on attractors}
Let $k$ be a field. Let $G$ be a smooth affine connected algebraic group over $k$. Let $X$ be a scheme of finite type over $k$ with a $G$-action. 

For any cocharacter $\lambda: \Gm \to G$, consider the functors of $\Gm$-equivariant morphisms
\[
X_\lambda:=\mathrm{Map}_k^{\Gm}(\Spec(k),X) \text{ and } X_\lambda^+:=\mathrm{Map}_k^{\Gm}(\bbA^1_k,X),
\]
where $\bbA^1_k=\Spec k[a]$ is the affine line such that $a$ has $\Gm$-weight $-1$. By \cite[Cor. 1.4.3]{Drinfeld}, these functors are representable by schemes of finite type over $k$. There are natural morphisms: a closed immersion (\cite[Prop. 1.2.2]{Drinfeld}) $X_\lambda \hookrightarrow X$, an unramified morphism $\ev_1: X_\lambda^+ \to X$ (evaluation at 1) and an affine (\cite[Thm. 1.4.2]{Drinfeld}) morphism $\ev_0: X_\lambda^+ \to X_\lambda$ (evaluation at 0) with a (closed immersion) section $\iota: X_\lambda \hookrightarrow X_\lambda^+$. 

For any connected component $F \subseteq X_\lambda$, define
\[
C_{\lambda,F}:=\ev_0^{-1}(F) \subseteq X_\lambda^+.
\]
Then $C_{\lambda,F}$ is connected, and all connected components of $X_\lambda^+$ are of this form.

In general, neither $X_\lambda$ nor $X_\lambda^+$ is $G$-invariant. We define
\[
L_\lambda:=\{l \in G: l=\lambda(t)l\lambda(t)^{-1}, \forall t\} \text{ and } P_\lambda^+:=\{p \in G: \exists \lim_{t \to 0} \lambda(t)p\lambda(t)^{-1}\}.
\]
If $G$ is geometrically reductive, then $P_\lambda^+ \subseteq G$ is a parabolic subgroup. There is a surjective homomorphism $P_\lambda^+ \twoheadrightarrow L_\lambda$ sending $p \mapsto \lim_{t \to 0} \lambda(t)p\lambda(t)^{-1}$. The scheme $X_\lambda$ inherits an action of $L_\lambda$ and $X_\lambda^+$ inherits an action of $P_\lambda^+$ such that $\ev_0: X_\lambda^+ \to X_\lambda$ is equivariant with respect to the surjection $P_\lambda^+ \twoheadrightarrow L_\lambda$, i.e., we have an induced morphism
\[
\ev_0: [X_\lambda^+/P_\lambda^+] \to [X_\lambda/L_\lambda].
\]
The groups $L_\lambda$ and $P_\lambda^+$ are connected, so they preserve connected components of $X_\lambda$ and $X_\lambda^+$.

\begin{remark}\label{rmk:ev1-closed-immersion}
If $X$ is separated, the evaluation map $\ev_1$ is a monomorphism \cite[Prop. 1.4.11]{Drinfeld}, but in general not a locally closed immersion, even restricted to connected components of $X_\lambda^+$ \cite[Ex. 1.4.12]{Drinfeld}. However, the restriction of $\ev_1$ to each connected component of $X_\lambda^+$ is a locally closed immersion if $X$ is smooth \cite[Thm. B.0.3 (i)]{Drinfeld}. 

If the quotient stack $[X/G]$ is $\Theta$-reductive (e.g., if $X$ admits a good quotient $X \to X/G$), then $\ev_1: X_\lambda^+ \to X$ is a closed immersion \cite[Prop. 3.13 and Rmk. 3.14]{AHLH}. In particular, for any $\lambda \in X_*(G)$ and $F \in \pi_0(X_\lambda)$, the restriction $\ev_1: C_{\lambda,F} \to X$ is a closed immersion, and in this case we identify $C_{\lambda,F}$ with a closed subscheme of $X$ via $\ev_1$, and it sits inside
\[
\begin{tikzcd}
C_{\lambda,F} \ar[r,hook] \ar[d] & X_\lambda^+ \ar[r,hook,"\ev_1"] \ar[d,"\ev_0"] & X \\
F \ar[r,hook] \ar[u,bend left,"\iota"] & X_\lambda \ar[u,bend left,"\iota"]
\end{tikzcd}
\]
\end{remark}
A priori, there may be infinitely many connected components $C_{\lambda,F}$. Nevertheless, if $X$ admits a good quotient, we show their $G$-saturations $G.C_{\lambda,F}$ (i.e., the set-theoretical image of the action morphism $G \times C_{\lambda,F} \to X$) form a finite collection.
\begin{setup}\label{setup:quotient}
Let $k$ be an algebraically closed field, let $G$ be a reductive group over $k$, and let $X$ be a $k$-variety with a $G$-action. Suppose that $X$ admits a good quotient $q: X \to X/G$.
\end{setup}
\begin{lemma}\label{lem:quotient-images}
Assume the hypotheses of  \Cref{setup:quotient}. For every cocharacter $\lambda \in X_*(G)$ and every connected component $F \in \pi_0(X_\lambda)$, the subset $G.C_{\lambda,F} \subseteq X$ is closed, and there are only finitely many distinct subsets of the form $G.C_{\lambda,F}$.
\end{lemma}
For projective varieties, finiteness of attracting decompositions as the cocharacter varies is established in
\cite[Prop. 1.3.8]{DolgachevHu1998}.
The affine finite-weight argument also appears in
\cite[proof of Lem. 4.4.8]{halpern-leistner2014instability}.
\begin{proof}
Fix $\lambda \in X_*(G)$ and $F \in \pi_0(X_\lambda)$. Since $X$ admits a good quotient, by \Cref{rmk:ev1-closed-immersion} we identify $C_{\lambda,F} \subseteq X_\lambda^+$ as a closed subscheme of $X$ via the closed immersion $\ev_1: X_\lambda^+ \hookrightarrow X$. 

For the closedness. The parabolic subgroup $P_\lambda^+ \subseteq G$ preserves $C_{\lambda,F}$. Consider the morphism
\[
\iota: G \times^{P^+_\lambda} C_{\lambda,F} \to G/P^+_\lambda \times X \text{ sending } [(g,x)] \mapsto (gP^+_\lambda,g.x). 
\]
This is a closed immersion. Indeed, after the faithfully flat base change $G \to G/P^+_\lambda$, it becomes $G \times C_{\lambda,F} \to G \times X$ sending $(g,x) \to (g,g.x)$, which is the closed immersion $\id_G \times \ev_1: G \times C_{\lambda,F} \hookrightarrow G \times X$ followed by an automorphism of $G \times X$. Since $G/P_\lambda^+$ is projective, the composition 
\[
G \times^{P^+_\lambda} C_{\lambda,F} \xrightarrow{\iota} G/P^+_\lambda \times X \xrightarrow{\mathrm{pr}_2} X
\]
is projective with set-theoretical image $G.C_{\lambda,F}$. This proves the closedness. 

For the finiteness. Fix a maximal torus $\bbT \subseteq G$. If $\mu=g\lambda g^{-1}$ for some $g\in G(k)$, then multiplication by $g$ induces isomorphisms $X_\lambda^+ \xrightarrow{\sim} X_\mu^+$ and $X_\lambda \xrightarrow{\sim} X_\mu$ compatible with evaluation maps, i.e., $g.C_{\lambda,F}=C_{\mu,gF}$. Thus we have $G.C_{\lambda,F}=G.C_{\mu,gF}$. Since every cocharacter of $G$ is conjugate to a cocharacter of $\bbT$, for the finiteness it suffices to consider $\lambda \in X_*(\bbT)$.

Suppose first that $X=\Spec A$ is affine. Choose homogeneous generators $f_1,\ldots,f_n$ of $A$ as a $k$-algebra, with weights $\chi_1,\ldots,\chi_n \in X^*(\bbT)$. Then by \cite[\S 1.3.4]{Drinfeld}
\[
X_\lambda^+ = \Spec A/I_\lambda^- \text{ with } I_\lambda^-:=
\langle f_i: \langle \lambda,\chi_i \rangle < 0 \rangle.
\]
This shows that there are finitely many distinct $X_\lambda^+$. Each $X_\lambda^+$ has finitely many connected components, so only finitely many subsets $C_{\lambda,F}$ occur.

For general $X$, by \cite[Thm. 4.1]{Luna-slice-for-stacks} and quasi-compactness of $X$, there exists a $\bbT$-equivariant \'{e}tale surjection $f: Y \twoheadrightarrow X$ with $Y$ affine. By the affine case, the collection $\cD:=\{Z \in \pi_0(Y_\mu^+): \mu \in X_*(\bbT)\}$ of subsets of $Y$ is finite. The induced morphism $f_\lambda^+=f|_{Y_\lambda^+}: Y_\lambda^+ \twoheadrightarrow X_\lambda^+$ is surjective by \cite[Prop. 5.20]{Luna-slice-for-stacks}. Hence
\[
C_{\lambda,F} = \bigcup_{Z \in \pi_0(Y_\lambda^+), f(Z) \subseteq C_{\lambda,F}} f(Z).
\]
Thus $C_{\lambda,F}$ is a union of images of elements in the finite collection $\cD$ under the fixed $f$. This proves the finiteness. The finiteness part does not need the existence of a good quotient.
\end{proof}
\begin{definition}
Let $\cL$ be a $G$-linearized line bundle on $X$. For any cocharacter $\lambda: \Gm \to G$ and any connected subset $F \subseteq X_\lambda$, we define 
\[
\wt_\lambda(\cL|_F):=\wt_\lambda(\cL|_x) \in \bbZ \text{ for some } x \in F.
\]
This is well-defined: the $\lambda$-weight, being discrete, is locally constant on $X_\lambda$, hence constant on the connected subset $F$. So the value of $\wt_\lambda(\cL|_F)$ does not depend on the choice of $x \in F$.

For a line bundle $\cL$ on $\Theta:=[\bbA^1/\Gm]$, we define $\wt(\cL):=\wt_\Gm(\cL|_0) \in \bbZ$. For any point $x \in C_{\lambda,F}$, its complete $\lambda$-orbit $\lambda_x: \bbA^1 \to X$ satisfies $\lambda_x(1)=x$ and $x_0:=\lambda_x(0)=\lim_{t \to 0} \lambda(t).x \in F$. This defines a morphism $\lambda_x: \Theta=[\bbA^1/\Gm] \to [X/G]$ and
\begin{equation}\label{eq:comparing-weight}
\wt(\lambda_x^*\cL) = \wt_\Gm(\lambda_x^*\cL|_0) = \wt_\lambda(\cL|_{x_0}) = \wt_\lambda(\cL|_F).   
\end{equation} 
\end{definition}
\begin{lemma}[\cite{Heinloth}, \S 1.A.a]\label{lem:theta-weight}
Let $K$ be a field, and let $\cL$ be a $\mathbb{G}_{m,K}$-equivariant line bundle on $\bbA^1_K$. Then
\[
\Gamma(\Theta_K,\cL)=\Gamma(\bbA^1_K,\cL)^{\mathbb{G}_{m,K}}=
\begin{cases}
0 & \text{if } \wt(\cL) < 0 \\
K \cdot a^{d}e & \text{if } \wt(\cL)=d \geq 0
\end{cases}
\]
where $\bbA^1_K=\Spec K[a]$ and $e$ is the unique (up to a non-zero scalar) non-vanishing section of $\cL$. In particular, there exists a non-zero section of $\cL$ vanishing at $0$ only if $\wt(\cL)>0$.
\end{lemma}
\begin{corollary}\label{cor:negative-in-D}
Let $\cX$ be an algebraic stack, let $\cD$ be an effective Cartier divisor on $\cX$, and let $\cL:=\cO_\cX(\cD)$. Then any morphism $f: \Theta \to \cX$ with $\wt(f^*\cL)<0$ factors through $\cD$.    
\end{corollary}
\begin{proof}
Since $\Gamma(\Theta,f^*\cL)=0$ by \Cref{lem:theta-weight}, for the canonical section $s$ of $\cL$ this means $f^*s=0$. Then $f$ factors through the zero locus of $s$, which is $\cD$.
\end{proof}
\begin{corollary}\label{cor:weight-inclusion}
Let $k$ be a field, let $G$ be an affine algebraic group over $k$, and let $X$ be a $k$-scheme with a $G$-action. Let $D$ be a $G$-invariant effective Cartier divisor on $X$ and $\cL:=\cO_X(D)$. For any $\lambda \in X_*(G)$ and $F \in \pi_0(X_\lambda)$, we have
\begin{enumerate}
\item 
If $\wt_\lambda(\cL|_F)<0$, then $G.C_{\lambda,F} \subseteq D$.
\item 
If $D \cap F = \varnothing$, then $\wt_{\lambda}(\cL|_F) = 0$.
\end{enumerate}
\end{corollary}
\begin{proof}
For (i), suppose $\wt_\lambda(\cL|_F)<0$. By $G$-invariance it suffices to show $C_{\lambda,F} \subseteq D$. For any $x \in C_{\lambda,F}$, let $\lambda_x: \bbA^1 \to X$ be its complete $\lambda$-orbit map so $\lambda_x(1)=x$ and $\lambda_x(0) \in F$, and it induces $\lambda_x: [\bbA^1/\Gm] \to [X/G]$. If $\wt_\lambda(\cL|_F)<0$, then $\wt(\lambda_x^*\cL)<0$ by \eqref{eq:comparing-weight}. This shows that $\lambda_x(1)=x \in D$ by \Cref{cor:negative-in-D}. For (ii), let $s \in \Gamma(X,\cL)^G$ be the canonical section of $\cL$ defining $D$. Since $D \cap F=\varnothing$, the section $s$ is nowhere vanishing on $F$. So it induces a $\lambda$-equivariant trivialization $\cO_F \xrightarrow{\sim} \cL|_F$ sending $1 \mapsto s|_F$. Since $\lambda$ acts trivially on $F$, it follows that $\wt_\lambda(\cL|_F)=0$.
\end{proof}
\begin{lemma}\label{lem:lifting-Theta}
Consider a Cartesian diagram
\[
\begin{tikzcd}
\cY \ar[r] \ar[d] & \cX \ar[d] \\
\cZ \ar[r] & X
\end{tikzcd}
\]
where $\cX,\cY,\cZ$ are algebraic stacks and $X$ is an algebraic space. A morphism $f: \Theta \to \cX$ lifts to $\cY$ if and only if $f(1) \in \cX$ lifts to $\cY$.
\end{lemma}
\begin{proof}
Consider $\Theta_k \xrightarrow{f} \cX \to X$. Since $X$ is an algebraic space, the composition factors through the good moduli space $\Theta_k \to \Spec k$, i.e., factors through a $k$-point $x \in X(k)$. If $f(1) \in \cX$ lifts to $\cY$, the $k$-point $x \in X(k)$ lifts to a $k$-point $z \in \cZ(k)$. The morphisms $f: \Theta_k \to \cX$ and $\Theta_k \to \Spec k \xrightarrow{z} \cZ$ induce a morphism $\Theta_k \to \cY$ lifting $f$.
\end{proof}
\subsection{Geometric characterization of the boundary}
\begin{theorem}\label{thm:quotient-boundary}
Assume the hypotheses of  \Cref{setup:quotient}. Let $D$ be a $G$-invariant reduced effective Cartier divisor on $X$ and $\cL:=\cO_X(D)$. Then $D \subseteq X$ is removable for $\pi$ if and only if
\begin{equation}\label{eq:character-D}
D = \bigcup_{\wt_\lambda(\cL|_F)<0} G.C_{\lambda,F}.
\end{equation}
In particular, if $D$ is irreducible, then $D \subseteq X$ is removable for $\pi$ if and only if
\[
D = G.C_{\lambda,F} \text{ with } \wt_\lambda(\cL|_F)<0.
\]
\end{theorem}
\begin{proof}
Since $X$ admits a good quotient, each subset $G.C_{\lambda,F}$ is closed in $X$, and they form a finite collection by \Cref{lem:quotient-images}. To show the equivalence we introduce a bridge condition

\vspace{0.3em}

(\S) for every geometric point $x$ of $D$, there exists a morphism $f: \Theta \to [D/G]$ with $f(1) \simeq x$ and $\wt(f^*\cN)<0$, where $\cN:=\cN_{D/X}=\cL|_D$ is the normal line bundle.

\vspace{0.3em}

Fix an \'{e}tale cover $W \to X/G$ with $W$ affine. Since $\pi$ is affine and $W$ is affine, its base change $X_W:=X\times_{X/G}W$ is affine. The following chain proves the equivalence
\[
\begin{tikzcd}
\text{the Cartier divisor } D \subseteq X \text{ is removable for } \pi \ar[d,"\text{by \Cref{thm:adequate-boundary}}"] \\
\Gamma(D_W,N_W^{\otimes n})^G=0 \text{ for every } n>0 \ar[u] \ar[d,"D_W \text{ is reduced}"] \\
D_W^\sst(N_W) = \varnothing \ar[d,"\text{by Hilbert-Mumford criterion}"] \ar[u] \\
(\S) \text{ on } D_W \ar[u] \ar[d,"\text{by \Cref{lem:lifting-Theta}}"] \\
(\S) \ar[u],
\end{tikzcd}
\]
where
\[
D_W:=D \times_{X/G} W \text{ and } N_W:=p_D^*\cN \simeq \cN_{D_W/X_W} \text{ with } p_D:D_W \to D \text{ the projection}
\]
The $G$-action on $D_W$ and the $G$-linearization of $N_W$ are induced by pullback, with $G$ acting trivially on $W$.

Suppose the condition (\S) holds. Then each geometric point of $D$ belongs to some $C_{\lambda,F}$ such that $\wt_\lambda(\cL|_F)<0$, so \Cref{cor:weight-inclusion} gives \eqref{eq:character-D}. Conversely, suppose the equality \eqref{eq:character-D} holds. Then each geometric point $x$ of $D$ can be written as $x=g.y$ for some $g \in G$ and $y \in C_{\lambda,F}$ such that $\wt_\lambda(\cL|_F)<0$. The cocharacter $g \lambda g^{-1}$ has a limit at $x$, lying in $gF \subseteq D$, with the same negative weight. This defines a morphism $f: \Theta \to [D/G]$ with $f(1) \simeq x$ and $\wt(f^*\cN)<0$. This proves (\S).
\end{proof}
\begin{corollary}\label{cor:finite-reductive}
Assume the hypotheses of  \Cref{setup:quotient}. Assume in addition that $X$ is smooth. Then $\PD(\pi)$ is finite, and each element of $\PD(\pi)$ is of the form
\[
G.C_{\lambda,F} \text{ for some } \lambda \in X_*(G), F \in \pi_0(X_{\lambda}) \text{ with } \wt_\lambda(\cO_X(G.C_{\lambda,F})|_F)<0.
\]
Write $\PD(\pi)=\{D_1,\ldots,D_r\}$. For a subset $\sigma \subseteq \{1,\ldots,r\}$, write 
\[
D_\sigma:=\sum_{i \in \sigma} D_i \text{ and } \cL_\sigma:=\cO_X(D_\sigma).
\]
Call $\sigma$ \emph{admissible} if, for every $i \in \sigma$, there exist $\lambda_i \in X_*(G)$ and $F_i \in \pi_0(X_{\lambda_i})$ such that
\[
D_i=G.C_{\lambda_i,F_i}  \text{ and } \wt_{\lambda_i}(\cL_\sigma|_{F_i})<0.
\]
The assignment $\sigma \mapsto X \setminus D_\sigma$ is an order-reversing bijection from admissible subsets of $\{1,\ldots,r\}$ to $\Op(q)$. In particular, we have $|\Op(q)| \leq 2^r$.
\end{corollary}
\begin{proof}
By definition, every element of $\PD(\pi)$ is removable and prime, hence of the form $G.C_{\lambda,F}$ for a single $\lambda \in X_*(G)$ and $F \in \pi_0(X_\lambda)$ with $\wt_\lambda(\cO_X(G.C_{\lambda,F})|_F)<0$ by \Cref{thm:quotient-boundary}. Such closed subsets $G.C_{\lambda,F}$ form a finite collection by \Cref{lem:quotient-images}. 

For every admissible subset $\sigma \subseteq \{1,\ldots,r\}$, the closed subscheme $D_\sigma$ is reduced and removable for $\pi$ by \Cref{thm:quotient-boundary}. This means the assignment $\sigma \mapsto X \setminus D_\sigma$ is well-defined and it is injective. Conversely, let $D = D_1 + \cdots + D_s$ be a reduced removable closed subscheme of $X$. Then $D$ is a $G$-invariant effective Cartier divisor on $X$ by \Cref{lem:boundary-purity}. Each prime boundary $D_i$ is $G$-invariant since $G$ is connected, and is removable for $\pi$ by \Cref{cor:removable-restriction}. Then $D_i \in \PD(\pi)$. 

Since $X$ is smooth, each connected attractor $C_{\lambda,F}$ is smooth by \cite[Prop. 1.4.20]{Drinfeld}, hence irreducible. As $G$ is connected, the $G$-saturation $G.C_{\lambda,F}$ is also irreducible. By \Cref{thm:quotient-boundary} and \Cref{lem:quotient-images}, $D$ is a finite union of these
closed irreducible subsets with $\wt_{\lambda}(\cO_X(D)|_F)<0$. Every prime component $D_i$ therefore equals one of the subsets occurring in this union, giving the required weight condition.

Counting possible subsets gives the upper bound.
\end{proof}

\begin{remark}\label{rmk:weight-computation-torus}
Fix a subset $\sigma \subseteq \{1,\ldots,r\}$. For any fixed index $i \in \sigma$, by \Cref{cor:finite-reductive}, there exist $\lambda_i \in X_*(G)$ and $F_i \in \pi_0(X_{\lambda_i})$ such that
\[
D_i=G.C_{\lambda_i,F_i} \text{ and } \wt_{\lambda_i}(\cO_X(D_i)|_{F_i})<0.
\]
Then we can compute
\begin{align*}
\wt_{\lambda_i}(\cL_\sigma|_{F_i})&=\wt_{\lambda_i}(\cO_X(D_\sigma)|_{F_i})=\wt_{\lambda_i}\left(\cO_X\left(\sum_{j \in \sigma} D_j\right)|_{F_i}\right) \\
&=\wt_{\lambda_i}\left(\bigotimes_{j \in \sigma} \cO_X(D_j)|_{F_i}\right) = \sum_{j \in \sigma} \wt_{\lambda_i}(\cO_X(D_j)|_{F_i}) \\
&=\underbrace{\wt_{\lambda_i}(\cO_X(D_i)|_{F_i})}_{<0 \text{ by } \ref{thm:quotient-boundary}}+\sum_{j \in \sigma \setminus \{i\}} \underbrace{\wt_{\lambda_i}(\cO_X(D_j)|_{F_i})}_{\geq 0 \text{ by } \ref{cor:weight-inclusion}}
\end{align*}
In particular, if $\wt_{\lambda_i}(\cL_\sigma|_{F_i})<0$ then $\wt_{\lambda_i}(\cO_X(D_i)|_{F_i})<0$. This explains \Cref{cor:removable-restriction} from the weight perspective. By \Cref{cor:finite-reductive}, we see $D_\sigma$ is removable if and only if for each $i \in \sigma$ the pair $(\lambda_i,F_i)$ can be arranged so that the $|\sigma|$-dimensional vector
\[
(\wt_{\lambda_i}(\cL_\sigma|_{F_i}))_{i \in \sigma} \in \bbZ^{|\sigma|}
\]
consists of negative components, if and only if the $|\sigma| \times |\sigma|$ matrix
\[
\left(\wt_{\lambda_i}(\cO_X(D_j)|_{F_i})\right)_{i,j \in \sigma} \in \Mat_{|\sigma|}(\bbZ)
\]
is a row strictly diagonally dominant (SDD) matrix.
\end{remark}
\begin{example}\label{ex:determinant}
Let $k$ be an algebraically closed field, and let $G:=\SL_2$ act
on $X:=\Mat_2$ by left multiplication. By the first fundamental theorem of invariant theory for $\SL_2$ (see, e.g., \cite[\S 3]{MR979587}, or \cite[Thm. 3.2]{MR485944}, or \cite[Thm. 3.1]{MR422314}) we have $k[X]^G=k[\det]$. Thus, taking determinants defines a good quotient $\pi=\det: X \to \bbA^1$. Consider the cocharacter $\lambda: \Gm \to G$ sending $t \mapsto \diag(t,t^{-1})$. We compute
\[
F = X_\lambda = \left\{
\begin{pmatrix}
0 & 0 \\
0 & 0
\end{pmatrix}
\right\} \text{ and } C_{\lambda,F} = \left\{
\begin{pmatrix}
a & b \\
0 & 0
\end{pmatrix}: a,b \in k
\right\}.
\]
Its $G$-saturation is the locus of rank $\leq 1$ matrices $D:=G.C_{\lambda,F} = \mathbf{V}(\det)$. This is a reduced prime effective Cartier divisor on $X$, but not removable for $\pi$. Indeed, since the determinant function $\det$ is $G$-invariant, we have
\[
\wt_\mu(\cO_X(D)|_E) = \wt_\lambda(\cO_X(D)|_F)=0
\]
for any $\mu \in X_*(G)$ and $E \in \pi_0(X_\mu)$ with $D=G.C_{\mu,E}$, as every non-trivial cocharacter of $\SL_2$ is conjugate to a positive multiple of $\lambda$, and every corresponding $G$-saturation is $D$. The negative-weight condition in \Cref{thm:quotient-boundary} fails.

The non-removability of $D$ can also be seen directly: removing $D$ gives the good quotient
\[
\pi=\det: X \setminus D \simeq \GL_2 \to \Gm.
\]
\end{example}
\subsection{Example of torus case}
As an explicit case illustrating \Cref{thm:quotient-boundary}, we examine the torus action. For a torus $\bbT$, the connected component $C_{\lambda,F}$ is $\bbT$-invariant. Moreover, the weight condition can be read off from the dimension of a smooth attracting cell, and this is the only additional input beyond the reductive case.
\begin{lemma}\label{lem:torus-cell}
Let $k$ be an algebraically closed field, let $\bbT$ be a torus over $k$, and let $X$ be a smooth $k$-variety with a $\bbT$-action. Suppose that $X$ admits a good quotient $q: X \to X/\bbT$. Let $\lambda \in X_*(\bbT)$ and $F \in \pi_0(X_\lambda)$. If $C_{\lambda,F}$ has codimension one in $X$, then it is a $\bbT$-invariant smooth prime Cartier divisor on $X$ and $\wt_\lambda(\cO_X(C_{\lambda,F})|_F)<0$.
\end{lemma}
\begin{proof}
The fixed and attracting loci of a smooth scheme are smooth \cite[Prop. 1.4.20]{Drinfeld}. Then the codimension-one connected component $C_{\lambda,F}$ is a smooth prime Cartier divisor on $X$. Choose a $\lambda$-fixed point $x \in F$. The tangent spaces at $x$ satisfy \cite[Prop. 1.4.11 (vi)]{Drinfeld}
\[
\rmT_xC_{\lambda,F}=(\rmT_xX)_{\lambda,\geq 0}.
\]
The fiber of the normal line bundle $\cO_X(C_{\lambda,F})|_{C_{\lambda,F}}$ at $x$ is the quotient $\rmT_xX/(\rmT_xX)_{\lambda,\geq 0}$, which is one-dimensional and has negative $\lambda$-weight. This finishes the proof.
\end{proof}

Hereafter we consider the $\Gm$ case. Let $k$ be an algebraically closed field, and let $X$ be a smooth $k$-variety with a $\Gm$-action. Suppose $X$ admits a good quotient $q: X \to X/\Gm$.  Every non-trivial cocharacter of $\Gm$ is a positive multiple of $\pm \id$, and a positive multiple does not change the fixed or attracting locus \cite[Prop. 5.24]{Luna-slice-for-stacks}. So it suffices to consider the cocharacters $\pm \id$. Denote their attracting loci by $X^\pm$, and write
\[
X^\Gm:=X_{\id}=X_{-\id}
\]
for their common fixed locus. Write
\[
X^\Gm = \bigsqcup_{i \in \pi_0(X^\Gm)} X_i \text{ and } X^\pm = \bigsqcup_{i \in \pi_0(X^\Gm)} X_i^\pm,
\]
for the decomposition into connected components with $X_i^\pm := (\ev_0^\pm)^{-1}(X_i)$. In this case, all possible $\bbT$-saturations $\bbT.C_{\lambda,F}=C_{\lambda,F}$ are precisely $\{X_i^\pm: i \in \pi_0(X^\Gm)\}$. By \Cref{cor:finite-reductive} and \Cref{lem:torus-cell} we have
\[
\PD(\pi)=\{X_i^+: \codim_X X_i^+=1\} \cup \{X_i^-: \codim_X X_i^-=1\}
\]
Let $\PD^\pm:=\{i \in \pi_0(X^\Gm): \codim_X X_i^\pm=1\} \subseteq \pi_0(X^\Gm)$.
\begin{theorem}\label{thm:main-Gm}
Let $k$ be an algebraically closed field, and let $X$ be a smooth $k$-variety with a $\Gm$-action. Suppose that $X$ admits a good quotient $q: X \to X/\Gm$. Then the assignment $(\sigma^+,\sigma^-) \mapsto X \setminus \cup_{i \in \sigma^\pm} X_i^\pm$ defines a bijection between
\[
\left\{ 
(\sigma^+,\sigma^-): \sigma^\pm \subseteq \PD^\pm, \sigma^+ \cap \sigma^-=\varnothing 
\right\} 
\xleftrightarrow{1:1} \Op(q).
\]
\end{theorem}
\begin{proof}
For any subset $\sigma \subseteq \{1,\ldots,r\}$, define $\sigma^\pm:=\sigma \cap \PD^\pm \subseteq \PD^\pm$. By \Cref{cor:finite-reductive}, we have a bijection
\[
\left\{
\sigma \subseteq \{1,\ldots,r\}: \wt_{\pm\id}(\cL_\sigma|_{X_i})<0 \text{ for } i \in \sigma^\pm
\right\}
\leftrightarrow 
\left\{
\Gm\text{-invariant opens of } X \text{ preserving } \pi
\right\}.
\]
It remains to show the weight condition on the left-hand side is equivalent to $\sigma^+ \cap \sigma^- = \varnothing$. Indeed, for each $i \in \PD^+ \cup \PD^-$, write
\[
\wt_{\pm \id}(\cO_X(X_i^+)|_{X_i})=\mp a_i \text{ and } \wt_{\pm \id}(\cO_X(X_i^-)|_{X_i})=\pm b_i,
\]
for some integers $a_i, b_i>0$. For $j \neq i$, the fixed component $X_i$
is disjoint from $X_j^\pm$. Hence, by \Cref{cor:weight-inclusion},
\[
\wt_{\pm \id}(\cO_X(X_j^\pm)|_{X_i}) = 0.
\]
Therefore, by \Cref{rmk:weight-computation-torus}, for each $i \in \sigma = \sigma^+ \cup \sigma^-$ we have
\begin{align*}
\wt_{\pm \id}(\cL_\sigma|_{X_i})&=\sum_{j \in \sigma^+}\wt_{\pm \id}(\cO_X(X_j^+)|_{X_i}) + \sum_{j \in \sigma^-} \wt_{\pm \id}(\cO_X(X_j^-)|_{X_i}) \\
&=\mp \mathbf{1}_{\sigma^+}(i) \cdot a_i \pm \mathbf{1}_{\sigma^-}(i) \cdot b_i.
\end{align*}
If $\sigma^+ \cap \sigma^- = \varnothing$, then
\[
\wt_{\id}(\cL_\sigma|_{X_i})=-a_i<0 \text{ for every } i \in \sigma^+,
\]
and
\[
\wt_{-\id}(\cL_\sigma|_{X_i})=-b_i<0 \text{ for every } i \in \sigma^-.
\]
Conversely, if $i \in \sigma^+ \cap \sigma^-$, the two required
inequalities are
\[
\wt_{\id}(\cL_\sigma|_{X_i})=-a_i+b_i<0 \text{ and } \wt_{-\id}(\cL_\sigma|_{X_i})=a_i-b_i<0,
\]
and they cannot hold simultaneously. Thus the weight inequalities $\wt_{\pm \id}(\cL_\sigma|_{X_i})<0$ hold for $i \in \sigma^\pm$ if and only if $\sigma^+ \cap \sigma^-=\varnothing$.
\end{proof}
\begin{example}\label{ex:axes}
Let $\bbT:=\mathbb{G}_m^r$ act on $\bbA^{2r}$ by $(t_i)_i.(x_i,y_i)_i=(t_ix_i,t_i^{-1}y_i)_i$. The good quotient is $q: \bbA^{2r} \to \bbA^r$ sending $(x_i,y_i)_i \mapsto (x_iy_i)_i$. For each $i$, let $\lambda_i: \Gm \to \bbT$ be the $i$-th coordinate cocharacter. Both $\pm \lambda_i$ have fixed locus $F_i=\mathbf{V}(x_i,y_i)$, and
\[
D_i^+:=G.C_{\lambda_i,F_i}=\mathbf{V}(y_i) \text{ and } D_i^-:=G.C_{-\lambda_i,F_i}=\mathbf{V}(x_i).
\]
A cocharacter with $s$ non-zero coordinates has attracting locus of codimension $s$, so these are the only divisorial attracting images. Ordering the signs as $+,-$, the weight matrix is
\[
\left(
\wt_{\varepsilon\lambda_i}
\bigl(\cO_X(D_j^\delta)|_{F_i}\bigr)
\right)_{(i,\varepsilon),(j,\delta)}
=
\mathbb{I}_r \otimes
\begin{pmatrix}
-1 & 1 \\
1 & -1
\end{pmatrix}.
\]
By \Cref{thm:quotient-boundary} and \Cref{cor:finite-reductive}, $\PD(q)=\{D_i^\pm: 1 \leq i \leq r\}$, and the preserving opens are exactly the complements of unions
containing at most one divisor from each pair $(D_i^+,D_i^-)$. Thus $|\PD(q)|=2r$ and $|\Op(q)|=3^r$.
\end{example}
The following corollary shows that \Cref{ex:axes} attains the maximal
number of preserving opens for a $\Gm$-action on affine space.

\begin{corollary}\label{cor:affine-space-three-opens}
Let $k$ be an algebraically closed field. Let $\Gm$ act on the affine space $X:=\bbA^n$ and let $q: X \to X/\Gm$ be the good quotient. Then $|\Op(q)| \leq 3$. The upper bound is sharp for every $n \geq 2$.
\end{corollary}
\begin{proof}
We first show the fixed locus $X^\Gm$ is non-empty and connected. If the $\Gm$-action on $\bbA^n$ is linear, this can be seen directly using coordinates. In general, choose a $\Gm$-equivariant closed immersion $X \hookrightarrow V$ into a finite-dimensional $\Gm$-representation. Choose a prime $\ell \neq \mathrm{char}(k)$ larger than the absolute values of all non-zero $\Gm$-weights of $V$. Then $V^{\mu_\ell}=V^\Gm$ and hence $X^{\mu_\ell}=X^\Gm$. Since $\mu_\ell(k)$ is cyclic of order $\ell$, applying \cite[Cor. 4.3]{Symonds-Gm} to $X$ gives
\[
\sum_{i \geq 0} \dim_{\bbF_\ell} H_c^i(X^\Gm,\bbF_\ell)=1,
\]
where $H_c^i$ denotes étale cohomology with compact supports. Each $Z \in \pi_0(X^\Gm)$ contributes a non-zero group $H_c^{2\dim Z}(Z,\bbF_\ell)$. The equality therefore implies that $X^\Gm \neq \varnothing$ and has exactly one connected component. 

By \Cref{thm:main-Gm}, the preserving
opens are $X$, or $X\setminus X_{\id}^+$ if $\codim_X X_{\id}^+=1$, or $X\setminus X_{\id}^-$ if $\codim_X X_{\id}^-=1$. The two attracting divisors cannot be removed simultaneously. This proves the upper bound. Sharpness follows from \Cref{ex:axes}.
\end{proof}
\section{Locally reductive stacks}\label{sec:locally-reductive}
To extend the results on quotient stacks to a broader class of stacks, we begin by observing that, for an adequate moduli space $\pi: \cX \to X$, the property of preserving $\pi$ can be checked \'{e}tale locally on $X$.
\begin{proposition}\label{prop:chart-test}
Let $\pi: \cX \to X$ be an adequate moduli space. Suppose that there is an \'{e}tale cover $\{X_i \to X\}_i$ of $X$ by affine schemes, and over each $X_i$ an isomorphism
\[
\cX \times_X X_i \simeq [V_i/G_i],
\]
where $V_i \to X_i$ is affine and $G_i \to X_i$ is a reductive group scheme. For an open substack $\cU \subseteq \cX$, let $W_i:=\cU \times_{\cX} V_i \subseteq V_i$. Then $\cU$ preserves $\pi$ if and only if
\begin{equation}\label{eq:chart-test}
W_i \text{ is affine, and } \Gamma(V_i,\cO_{V_i})^{G_i} \xrightarrow{\sim} \Gamma(W_i,\cO_{W_i})^{G_i} \text{ for each } i.
\end{equation}
Conversely, a family of $G_i$-invariant open subsets $W_i \subseteq V_i$ satisfying \eqref{eq:chart-test} uniquely determines an open substack of $\cX$ preserving $\pi$, provided these open subsets $W_i \subseteq V_i$ agree after pullback to each $V_i \times_{\cX} V_j$.
\end{proposition}
\begin{proof}
Suppose an open substack $\cU \subseteq \cX$ preserves $\pi$. Then $\cU \hookrightarrow \cX$ is affine by \Cref{lem:affine-open} and hence the composition $\cU \hookrightarrow \cX \to X$ is adequately affine. Consider the following Cartesian diagram
\[
\begin{tikzcd}
W_i \ar[r,two heads] \ar[d,hook] & \left[ W_i/G_i \right] \ar[r] \ar[d,hook] & \cU \ar[d,hook] \\
V_i \ar[r,two heads] & \left[ V_i/G_i \right] \ar[r] \ar[d,"\pi_i"',two heads] & \cX \ar[d,"\pi",two heads] \\
& X_i \ar[r] & X
\end{tikzcd}
\]
The $G_i$-torsor $W_i \twoheadrightarrow [W_i/G_i]$ is affine as $G_i$ is affine, and $[W_i/G_i] \to X_i$ is an adequate moduli space as a flat base change of $\pi|_\cU: \cU \to X$. So the composition $W_i \to X_i$ is adequately affine, hence affine by \cite[Thm. 4.3.1]{AlperAdequate}. In particular, $W_i$ is affine. The condition that $[W_i/G_i] \to X_i$ is an adequate moduli space is precisely the isomorphism in \eqref{eq:chart-test}, using the isomorphism $\Gamma(X_i,\cO_{X_i}) \xrightarrow{\sim} \Gamma(V_i,\cO_{V_i})^{G_i}$.

Conversely, since $W_i$ is affine, the morphism $[W_i/G_i] \to \Spec \Gamma(W_i,\cO_{W_i})^{G_i}$ is an adequate moduli space \cite[Rmk. 9.1.5]{AlperAdequate}, and the isomorphism in \eqref{eq:chart-test} identifies that with $X_i$. So each open substack $[W_i/G_i] \subseteq [V_i/G_i]$ preserves $\pi_i$. This compatible family of open substacks $[W_i/G_i]$ descends uniquely along the \'{e}tale cover $\{[V_i/G_i] \to \cX\}$ to an open substack $\cU \subseteq \cX$, whose pullbacks to each $V_i$ are the prescribed $W_i$. The open substack $\cU \subseteq \cX$ preserves $\pi$ as adequate moduli spaces descend along any fpqc cover \cite[Prop. 5.2.9 (2)]{AlperAdequate}.
\end{proof}
We already know that preserving opens can be checked \'{e}tale-locally, and that for quotient stacks we can characterize preserving opens (\Cref{thm:quotient-boundary}). It is therefore natural to expect that the conclusion generalizes to algebraic stacks that are \'{e}tale-locally quotient stacks; this is precisely the notion of being locally reductive, which we briefly recall here. 
\begin{definition}[Local reductivity, \cite{AHLH}, Defns. 2.1 and 2.5]\label{def:locally-reductive}
A quasi-separated algebraic stack $\cX$ with affine stabilizers is called \emph{locally reductive} if every point of $\cX$ specializes to a closed point and every closed point $x$ admits an \'etale quotient presentation, which is a pointed \'{e}tale morphism $([\Spec A/\GL_n],w) \to (\cX,x)$ that induces an isomorphism of stabilizer groups at $w$.
\end{definition}
\begin{remark}
The technical condition ``local reductivity'' emerges due to the lack of a local structure theorem of algebraic stacks in positive characteristic. For an algebraic stack $\cX$ of finite type over a field, with affine stabilizers and admitting an adequate moduli space, local reductivity is automatic in the following cases:
\begin{itemize}
\item 
The adequate moduli space of $\cX$ is actually a good moduli space (e.g. if $\cX$ is of characteristic $0$).
\item
The stack $\cX$ admits a representable morphism to a quotient stack for an action of a reductive group. Indeed, this is covered in the proof of \cite[Prop. 3.28]{dario-zhang}.
\end{itemize}
\end{remark}
The following property of locally reductive stacks is precisely what allows us to extend the results on quotient stacks to a broader class of stacks.
\begin{lemma}\label{lem:adequate-charts}
Let $\cX$ be a locally reductive algebraic stack with separated diagonal. Suppose that there is an adequate moduli space $\pi: \cX \to X$. There is an \'{e}tale cover $\{\Spec B_i \to X\}_i$ of $X$ by affine schemes and Cartesian diagrams
\[
\begin{tikzcd}
\left[ \Spec A_i/\GL_{n_i} \right] \ar[r] \ar[d,"\pi_i"'] & \cX \ar[d,"\pi"] \\
\Spec B_i \ar[r] & X
\end{tikzcd}
\]
where $\pi_i$ is an adequate moduli space, i.e., $A_i^{\GL_{n_i}} = B_i$. If $\cX$ is quasi-compact, the \'{e}tale cover of $X$ can be chosen finite.
\end{lemma}
\begin{proof}
For each closed point $z \in |\cX|$, there is an \'{e}tale neighborhood $g_z: \Spec B_z \to X$ of $\pi(z)$ and a Cartesian diagram as above \cite[Prop. 2.10]{AHLH}. It remains to show the morphisms $g_z$ cover $X$. Let $x \in |X|$. Since $\pi$ is surjective, choose $\xi \in |\cX|$ with $\pi(\xi)=x$. By local reductivity of $\cX$, the point $\xi$ specializes to a closed point $z \in |\cX|$, so $x$ specializes to $\pi(z)$. The image of $g_z$ is open and contains $\pi(z)$, it therefore contains its generalization $x$. Note that $\cX$ is quasi-compact if and only if $X$ is quasi-compact. The finiteness follows.
\end{proof}
Therefore, every removable closed substack of a locally reductive algebraic stack $\cX$ is, \'{e}tale-locally, of the form $G.C_{\lambda,F}$ (\Cref{prop:chart-test}, \Cref{lem:adequate-charts} and \Cref{thm:quotient-boundary}). It remains to identify what global object is obtained by gluing these local $G$-saturations (possibly for different $G$). It turns out that the correct description uses gradings and filtrations of stacks.
\subsection{Recollections on gradings and filtrations}
Let $k$ be a field. Define $\Theta_k:=[\bbA^1_k/\mathbb{G}_{m,k}]$ and let $0:=[0/\mathbb{G}_{m,k}]=\mathrm{B}\mathbb{G}_{m,k} \in [\bbA^1_k/\mathbb{G}_{m,k}]$ be the unique closed point. For an algebraic stack $\cX$ of finite type over $k$ with affine stabilizers, define
\[
\Grad(\cX):=\Map_k(\mathrm{B}\mathbb{G}_{m,k},\cX) \text{ and } \Filt(\cX):=\Map_k(\Theta_k,\cX).
\]
These are algebraic stacks locally of finite type over $k$ \cite[Thm. 5.10]{Luna-slice-for-stacks}. For a line bundle $\cL$ on $\cX$ and a connected subset $Z \subseteq \Grad(\cX)$, define
\[
\wt(\cL|_Z):=\wt_\Gm(x^*\cL) \in \bbZ \text{ for some } (x: \mathrm{B}\Gm \to \cX) \in |Z|.
\]
This is well-defined: the discrete $\Gm$-weight is locally constant on $\Grad(\cX)$, hence constant on the connected subset $Z$. In particular, the value of $\wt_\Gm(x^*\cL)$ does not depend on the choice of $x \in |Z|$. There are maps
\[
\begin{tikzcd}
\Grad(\cX) \ar[r,bend right,"\sigma"'] & \Filt(\cX) \ar[r,"\ev_1"] \ar[l,"\gr"'] & \cX,
\end{tikzcd}
\]
given by restriction to $0$ and evaluation at $1$ respectively. Any grading $\mathrm{B}\Gm \to \cX$ has a split filtration, obtained by precomposing with the canonical projection $\Theta \to \mathrm{B}\Gm$, and this yields a section $\sigma: \Grad(\cX) \to \Filt(\cX)$ of $\gr$. 

Multiplication of the affine coordinate by a parameter contracts any filtration to its split filtration. It follows that the map $\gr: \Filt(\cX) \to \Grad(\cX)$ induces a bijection on connected components. More explicitly, for a connected component $Z \subseteq \Grad(\cX)$, every point of $\gr^{-1}(Z)$ is connected, by that contraction, to a point of $\sigma(Z)$. Hence $\gr^{-1}(Z)$ is connected. Define a subset
\begin{equation}\label{eq:AZ}
|\cA_Z| := \ev_1(|\gr^{-1}(Z)|) \subseteq |\cX|.
\end{equation}
If it is closed, let $\cA_Z$ denote the unique reduced closed substack of $\cX$ with that support. 

To see why \eqref{eq:AZ} is the correct global object, we check the case of quotient stacks.
\begin{example}
Let $k$ be a field. Let $X$ be a scheme of finite type over $k$ with an action of a smooth affine connected algebraic group $G$ over $k$ with a split maximal torus. By \cite[Thm. 1.4.8]{halpern-leistner2014instability} we have
\begin{equation}\label{eq:quotient-filtrations}
\Grad([X/G]) \simeq \coprod_{\lambda} [X_\lambda/L_\lambda] \text{ and } \Filt([X/G]) \simeq \coprod_{\lambda}[X_\lambda^+/P_\lambda^+]
\end{equation}
where $\lambda$ runs through a complete set of conjugacy classes of $X_*(G)$. In particular, any connected component $Z \subseteq \Grad([X/G])$ is of the form $[F/L_\lambda]$ for some cocharacter $\lambda \in X_*(G)$ and connected component $F \in \pi_0(X_\lambda)$. Therefore
\[
\gr^{-1}(Z)=\gr^{-1}([F/L_\lambda])=[C_{\lambda,F}/P_\lambda^+],
\]
and if $X$ admits a good quotient $X \to X/G$ (cf. \Cref{rmk:ev1-closed-immersion}), then
\[
|\cA_Z| = \ev_1(|\gr^{-1}(Z)|) = \ev_1([C_{\lambda,F}/P_\lambda^+]) = [G.C_{\lambda,F}/G] \subseteq [X/G].
\]
\end{example}
With these in hand, we can globalize the boundary criterion and finiteness result for quotient stacks in \Cref{sec:reductive} to locally reductive stacks.
\begin{setup}\label{setup:locally-reductive}
Let $k$ be an algebraically closed field, let $\cX$ be a locally reductive algebraic stack of finite type over $k$ with separated
diagonal, and let $\pi: \cX \to X$ be an adequate moduli space.
\end{setup}

\begin{lemma}[Parallel to \Cref{lem:quotient-images}]\label{prop:finite-images}
Assume the hypotheses of \Cref{setup:locally-reductive}. For every connected component $Z \in \pi_0(\Grad(\cX))$, the subset $|\cA_Z| \subseteq |\cX|$ is closed, and there are only finitely many distinct subsets of the form $|\cA_Z|$.
\end{lemma}
\begin{remark}
The finiteness assertion also follows from the reduction to affine torus quotients in \cite[Lem. 4.4.6 and proof of Lemm. 4.4.8]{halpern-leistner2014instability}. More generally, finiteness results for components of graded and filtered stacks, under the very mild hypothesis of quasi-compact graded points, are proved in \cite[Thms. 5.4.6 and 6.2.3]{BHIK2025}.
\end{remark}
\begin{proof}
By \Cref{lem:adequate-charts}, there is an \'{e}tale cover $\{X_i \to X\}_i$ of $X$ by finitely many affine schemes and Cartesian presentations $\rho_i: \cX_i:=[V_i/\GL_{n_i}] \to \cX$ with $V_i$ affine and $n_i \geq 0$. The induced morphism $V_i \to \cX_i \to X_i$ is a good quotient.

For every $k$-scheme $S$, the morphisms $\Theta_S \to S$ and
$\mathrm{B}\mathbb{G}_{m,S} \to S$ are good moduli spaces and hence universal for maps to algebraic spaces \cite[Thm. 6.6]{AlperGood}. So we obtain
\[
\Filt(\cX_i) \simeq \Filt(\cX) \times_X X_i \text{ and } \Grad(\cX_i) \simeq \Grad(\cX) \times_X X_i,
\]
compatible with $\gr$ and $\ev_1$. For a connected component $Z \in \pi_0(\Grad(\cX))$, we have a Cartesian diagram
\[
\begin{tikzcd}
Z_i \ar[r,hook] \ar[d] & \Grad(\cX_i) \ar[d] \ar[r,hook,"\sigma_i"] & \Filt(\cX_i) \ar[r,"\ev_1"] \ar[d] & \cX_i \ar[r,"\pi_i"] \ar[d,"\rho_i"] & X_i \ar[d,"\text{\'{e}tale}"] \\
Z \ar[r,hook] & \Grad(\cX) \ar[r,hook,"\sigma"'] & \Filt(\cX) \ar[r,"\ev_1"'] & \cX \ar[r,"\pi"'] & X
\end{tikzcd}
\]
The pullback $Z_i:=Z \times_X X_i$ is a union of connected components of $\Grad(\cX_i)$. Since topological images commute with base change, we obtain 
\begin{equation}\label{eq:lr-evaluation-base-change}
\rho_i^{-1}(|\cA_Z|) = \bigcup_{Z' \in \pi_0(Z_i)}|\cA_{Z'}|.
\end{equation}
By \eqref{eq:quotient-filtrations}, each connected component $Z'$ is of the form $[F/L_\lambda]$, and the preimage of $|\cA_{Z'}| \subseteq |\cX_i|$ in $V_i$ is $\GL_{n_i}.C_{\lambda,F}$. \Cref{lem:quotient-images} shows that there are only finitely many distinct $\GL_{n_i}$-saturations and each of them is closed in $V_i$. Thus the right-hand side of \eqref{eq:lr-evaluation-base-change} is a finite union of closed subsets in $|\cX_i|$. Closedness descends along the \'{e}tale cover $\{\cX_i \to \cX\}_i$, so $|\cA_Z| \subseteq |\cX|$ is closed. On each \'{e}tale chart $\cX_i \to \cX$ there are only finitely many possible unions, and the finite cover determines a global subset $|\cA_Z|$ by its pullbacks, proving finiteness.
\end{proof}
Keep the notation in the proof of \Cref{prop:finite-images}. Let $\cL$ be a line bundle on $\cX$, and let $\cL_i$ be its pullback along $V_i \to \cX_i=[V_i/\GL_{n_i}] \to \cX$. By \eqref{eq:quotient-filtrations} every connected component $Z' \in \pi_0(Z_i)$ is of the form $[F/L_\lambda]$ for some $\lambda \in X_*(\GL_{n_i})$ and $F \in \pi_0((V_i)_\lambda)$. Hence we have 
\begin{equation}\label{eq:lr-weight-comparison} \wt(\cL|_Z) = \wt(\rho_i^*\cL|_{Z'}) = \wt_\lambda(\cL_i|_F). 
\end{equation}
\begin{corollary}[Parallel to \Cref{cor:weight-inclusion}]
Assume the hypotheses of \Cref{setup:locally-reductive}. Let $\cD$ be an effective Cartier divisor on $\cX$ and $\cL:=\cO_{\cX}(\cD)$. For any $Z \in \pi_0(\Grad(\cX))$, we have
\begin{enumerate}
\item 
If $\wt(\cL|_Z)<0$, then $\cA_Z \subseteq \cD$.
\item 
Let $e: \Grad(\cX) \xrightarrow{\sigma} \Filt(\cX) \xrightarrow{\ev_1} \cX$ be the morphism forgetting the grading. If $e(|Z|) \cap |\cD|=\varnothing$, then $\wt(\cL|_Z)=0$.
\end{enumerate}
\end{corollary}
\begin{proof}
Let $s \in \Gamma(\cX,\cL)$ be the canonical section whose zero substack is $\cD$.

(i). Suppose that $\wt(\cL|_Z)<0$. Let $K/k$ be an algebraically closed field, and let $f: \Theta_K \to \cX$ be a filtration whose associated grading $f|_{\mathrm{B}\mathbb{G}_{m,K}}$ belongs to $Z$. By definition $\wt(f^*\cL) = \wt(\cL|_Z) < 0$. So $f$ factors through $\cD$ by \Cref{cor:negative-in-D}. Consequently, $\ev_1(f)$ belongs to $\cD$. Taking all such filtrations, and using the definition $|\cA_Z| = \ev_1 \left( |\gr^{-1}(Z)| \right)$, we obtain $\cA_Z \subseteq \cD$.

(ii). Suppose that $e(|Z|) \cap |\cD| = \varnothing$. Let $g: \mathrm{B}\mathbb{G}_{m,K} \to \cX$ be any grading in $Z$. The underlying topological space of $\mathrm{B}\mathbb{G}_{m,K}$ consists of one point, whose image under $g$ is $e([g])$. Our assumption therefore implies that $g$ factors through $\cX \setminus \cD$. The canonical section $s$ is nowhere vanishing on this complement. Its pullback defines an isomorphism $\cO_{\mathrm{B}\mathbb{G}_{m,K}} \xrightarrow{\sim} g^*\cL$, sending $1 \mapsto g^*s$, of line bundles on $\mathrm{B}\mathbb{G}_{m,K}$. Hence $\wt(\cL|_Z)=\wt(g^*\cL)=0$, as required.
\end{proof}

\subsection{Geometric characterization of the boundary}
\begin{theorem}[Parallel to \Cref{thm:quotient-boundary}]\label{thm:geometric-good}
Assume the hypotheses of \Cref{setup:locally-reductive}. Let $\cD$ be a reduced effective Cartier divisor on $\cX$ and $\cL:=\cO_{\cX}(\cD)$.
Then $\cD \subseteq \cX$ is removable for $\pi$ if and only if
\begin{equation}\label{eq:description-in-locally-reductive}
\cD = \bigcup_{Z \in \pi_0(\Grad(\cX)), \wt(\cL|_Z)<0} \cA_Z.
\end{equation}
In particular, if $\cD$ is irreducible, then $\cD \subseteq \cX$ is removable for $\pi$ if and only if 
\[
\cD = \cA_Z \text{ with } \wt(\cL|_Z)<0.
\]
\end{theorem}

\begin{proof}
Use the \'{e}tale charts in the proof of \Cref{prop:finite-images}. Write $D_i$ for the pullback of $\cD$ to $V_i$.  By \Cref{prop:chart-test}, $\cD$ is removable for $\pi$ if and only if each $D_i$ is removable for $\pi_i$. By \Cref{thm:quotient-boundary}, the latter condition is equivalent to
\[
D_i = \bigcup_{\wt_\lambda(\cO_{V_i}(D_i)|_F)<0}        \GL_{n_i}.C_{\lambda,F}.
\]
The equations \eqref{eq:lr-evaluation-base-change} and
\eqref{eq:lr-weight-comparison} identify these equalities with the
pullbacks of \eqref{eq:description-in-locally-reductive}, so they are equivalent to \eqref{eq:description-in-locally-reductive}.
The finiteness assertion follows from \Cref{prop:finite-images}.
\end{proof}
\begin{corollary}[Parallel to \Cref{cor:finite-reductive}]\label{cor:good-finite}
Assume the hypotheses of \Cref{setup:locally-reductive}. Assume in addition that $\cX$ is smooth. Then $\PD(\pi)$ is finite, and each element of $\PD(\pi)$ is of the form
\[
\cA_Z \text{ with } \wt(\cO_{\cX}(\cA_Z)|_Z)<0.
\]
Write $\PD(\pi)=\{\cD_1,\ldots,\cD_r\}$. For a subset $\sigma \subseteq \{1,\ldots,r\}$, write
\[
\cD_\sigma:=\sum_{i \in \sigma} \cD_i \text{ and } \cL_\sigma:=\cO_{\cX}(\cD_\sigma).
\]
Call $\sigma$ \emph{admissible} if, for every $i \in \sigma$, there exists $Z_i \in \pi_0(\Grad(\cX))$ such that
\begin{equation}\label{eq:compatibility-weights}
\cD_i=\cA_{Z_i} \text{ and } \wt(\cL_\sigma|_{Z_i})<0.
\end{equation}
The assignment $\sigma \mapsto \cX \setminus \cD_\sigma$ is an order-reversing bijection from admissible subsets of $\{1,\ldots,r\}$ to $\Op(\pi)$. In particular, we have $|\Op(\pi)| \leq 2^r$.
\end{corollary}
\begin{proof}
Since $\cX$ is smooth, the stack $\Filt(\cX)$ is smooth by \cite[Prop. 1.3.1 (4)]{halpern-leistner2014instability}. For each connected component
$Z \subseteq \Grad(\cX)$, the preimage $\gr^{-1}(Z)$ is smooth and connected, hence irreducible. Consequently, its image $\cA_Z$ is irreducible.

By \Cref{thm:geometric-good} and \Cref{prop:finite-images}, every removable reduced Cartier boundary $\cD_\sigma$ is a finite union of
closed irreducible substacks $\cA_Z$ satisfying $\wt(\cL_\sigma|_Z)<0$.
Each prime component $\cD_i$ thus equals one of these $\cA_Z$, proving \eqref{eq:compatibility-weights}. The remaining argument is identical to that of \Cref{cor:finite-reductive}.
\end{proof}
\begin{remark}
In characteristic zero, it is possible to prove the finiteness of preserving opens using \cite[Cor. 4]{Benoist2013} (under normality in place of smoothness) and passage to \'{e}tale quotient charts.
\end{remark}
\begin{example}[A log-pair wall crossing]\label{ex:log-wall}
We work over the field $\bbC$ of complex numbers. Following \cite{ADL}, let $\overline{\cP}^{\GIT}_4$ denote the quotient stack of GIT-semistable plane quartics by $\PGL_3$, and let $\overline{\cP}^{\mathrm{K}}_{4,3/8}$ denote the stack of K-semistable log Fano pairs admitting a $\bbQ$-Gorenstein smoothing to 
\[
\left( \PP^2,\dfrac{3}{8}C \right)
\]
with $C$ a smooth plane quartic. Their good moduli spaces are denoted by $\overline{P}^{\GIT}_4$ and $\overline{P}^{\mathrm{K}}_{4,3/8}$, respectively. By \cite[Thm. 5.5]{ADL}, the first wall crossing gives an open immersion $\Phi^-: \overline{\cP}^{\GIT}_4 \hookrightarrow \overline{\cP}^{\mathrm K}_{4,3/8}$ inducing an isomorphism of good moduli spaces $\phi^-: \overline{P}^{\GIT}_4 \xrightarrow{\sim} \overline{P}^{\mathrm{K}}_{4,3/8}$. We describe its boundary using \Cref{thm:geometric-good}.

Let $\cX:=\overline{\cP}^{\mathrm{K}}_{4,3/8}$, identify $\cU:=\overline{\cP}^{\GIT}_4$ with its image, and write $\cD:=(\cX \setminus \cU)_{\red}$. The geometric points of $\cD$ are represented by
\[
\left( \PP(1,1,4),\dfrac{3}{8} C_f \right) \text{ with } C_f=(z^2+f_8(x,y)=0) \text{ and } f_8 \in H^0(\PP^1,\cO_{\PP^1}(8)),
\]
where $x,y,z$ have weights $1,1,4$. The distinguished K-polystable point is
\[
q=\left( \PP(1,1,4),\dfrac{3}{8}Q'_4 \right) \text{ and } Q'_4=(z^2=0).
\]
The cocharacter
\[
\lambda: \Gm \to \Aut(q) \text{ given by } \lambda(t)[x:y:z]=[x:y:tz],
\]
gives the degeneration $z^2 + t^2 f_8(x,y)=0$ to $q$, as in \cite[Lem. 5.3 and its proof]{ADL}. More explicitly, setting
\[
H:=\Aut(q)=\GL_2 \times \Gm / \{(a\mathrm{I}_2,a^4): a \in \Gm\},
\]
we have $\cD \simeq [H^0(\PP^1,\cO_{\PP^1}(8))/H]$, with $\lambda$ acting on $H^0(\PP^1,\cO_{\PP^1}(8))$ with weight $2$. Let $Z \subseteq \Grad(\cX)$ be the connected component containing the grading $(q,\lambda)$. Then $Z \simeq \mathrm{B}H$.

The weight of the normal line bundle can be read from the smoothing $u_0u_2-u_1^2 = sz \subseteq \PP(1,1,1,2)$. The boundary is locally defined by $s=0$, so $\wt(\cO_{\cX}(\cD)|_Z) = -1$. Every point of $\cD$ has a degeneration with associated graded in $Z$, and the negative boundary weight gives $\cA_Z \subseteq \cD$. Consequently, $\cA_Z=\cD$ and $\wt(\cO_{\cX}(\cD)|_Z)<0$. Thus this wall-crossing realizes a prime removable boundary in our classification.  Moreover, a computation of the attracting images shows that $\cD$ is the unique codimension-one attracting image of $\cX=\overline{\cP}^{\mathrm K}_{4,3/8}$; hence, by \Cref{cor:good-finite}, $\overline{\cP}^{\GIT}_4$ is the unique non-trivial open substack of $\cX$ preserving its good moduli space.
\end{example}

\section{Bounding preserving opens by the relative Picard rank}\label{sec:bound}
By \Cref{cor:good-finite}, the number of preserving opens is finite for locally factorial, locally reductive stacks. In this section, we give a precise bound in terms of the relative Picard rank, when the stack has stable points in the following sense.
\begin{definition}[\cite{MR4255045}, Defn. 2.5 and \cite{fernandez_herrero_distinguishing_2024}, Defn. 5.8]\label[definition]{defn:stable_adequate_moduli}
Let $\pi: \cX \to X$ be an adequate moduli space. A point $x$ of $\cX$ is \emph{stable} relative to $\pi$ if $|\pi^{-1}(\pi(x))|=\{x\}$ under the induced map of topological spaces $|\cX| \to |X|$. We say $\pi$ is a \emph{stable} adequate moduli space if there is an open dense set of stable points in $\cX$.
\end{definition}
By \cite[Thm. 5.3.1 (5)]{AlperAdequate}, the stable points of $\cX$ are exactly those points that are not identified with other points when passing to the adequate moduli space $X$. For an abelian group $A$, we denote by $\rk A:=\dim_{\bbQ} A \otimes_{\bbZ} \bbQ$ its rank.

\begin{definition}
Let $\pi: \cX \to X$ be an adequate moduli space. The \emph{relative Picard group} is (as an abstract abelian group)
\[
\Pic(\cX/X):=\coker \bigl( \pi^*: \Pic(X) \to \Pic(\cX) \bigr)
\]
and denote its rank by $\rho_\pi:=\rk \Pic(\cX/X)$.
\end{definition}
\begin{remark}
By the projection formula and $\pi_*\cO_\cX = \cO_X$, the pullback $\pi^*: \Pic(X) \to \Pic(\cX)$ is injective. Note that $[\cL] \in \Pic(\cX/X)$ is torsion precisely when some positive power of $\cL$ descends to $X$.
\end{remark}
We begin by showing that the relative Picard rank $\rho_\pi$ is finite. For this we need the following descent lemma. Recall that for good moduli spaces, a vector bundle with stabilizers acting trivially on the closed fibers descends to the moduli space \cite[Thm. 10.3]{AlperGood}. This fails for adequate moduli spaces \cite[Ex. 5.6.1]{AlperAdequate}, but we show that descent still holds up to a power, as  expect.
\begin{lemma}[Descent up to a power]
\label{lem:adequate-power-descent}
Let $k$ be an algebraically closed field, and let $\pi: \cX \to X$ be an adequate moduli space, where $\cX$ is an algebraic stack of finite type over $k$. Let $\cL$ be a line bundle on $\cX$. Suppose that, for every geometric point $x \in |\cX|$ with closed image, the stabilizer $G_x$ acts trivially on $\cL_x$. Then there exists an integer $n>0$ such that $\cL^{\otimes n} \simeq \pi^*M$ for some line bundle $M$ on $X$.
\end{lemma}

\begin{proof}
Adequate moduli spaces are preserved by flat base change \cite[Prop. 5.2.9 (1)]{AlperAdequate}, and descent of a line bundle along an adequate moduli space is fpqc-local on the target \cite[Lem. 5.1.5]{MR5066465}. Since $X$ is quasi-compact, we may work over an \'{e}tale cover of $X$ by finitely many affine schemes and take a common multiple of the resulting exponents. So hereafter we may assume that $X$ is affine.

Fix a closed point $y \in X$, and let $x$ be the unique closed point of $\pi^{-1}(y)$ (see \cite[Thm. 5.3.1]{AlperAdequate}). Since $y$ is closed, $x$ is also closed in $\cX$. Its residual gerbe $i: \mathrm{B}G_x \hookrightarrow \cX$ is therefore a closed immersion. By hypothesis, $i^*\cL$ is trivial. Choose a nowhere-vanishing section $t$ of $i^*\cL$. Consider the surjection of graded quasi-coherent $\cO_\cX$-algebras
\[
\bigoplus_{d \geq 0} \cL^{\otimes d} \twoheadrightarrow i_*\bigoplus_{d \geq 0}i^*\cL^{\otimes d}.
\]
By adequacy, some positive power of $t$ lifts to a global section of $\bigoplus_{d \geq 0} \cL^{\otimes d}$. Taking its homogeneous component of the appropriate degree gives an integer $n_y>0$ and a section (see \cite[Lem. 4.1.8]{AlperAdequate}) $s_y \in \Gamma(\cX,\cL^{\otimes n_y})$ such that $i^*s_y=t^{n_y}$. The zero locus $\mathbf{V}(s_y)$ does not contain $x$, and hence is disjoint from $\pi^{-1}(y)$. Since $\pi$ is closed, the subset $U_y:=X \setminus \pi(\mathbf{V}(s_y)) \subseteq X$ is an open neighborhood of $y$, and $s_y$ trivializes $\cL^{\otimes n_y}$ on $\pi^{-1}(U_y)$. These neighborhoods cover all closed points of $X$, and hence cover $X$. Choose a finite subcover $U_{y_1},\dots,U_{y_r}$ of $X$, and let $n$ be a common multiple of $n_{y_1},\dots,n_{y_r}$. Then $\cL^{\otimes n}$ is trivial over every open $\pi^{-1}(U_{y_i})$, so it descends to $X$ by \cite[Lem. 5.1.5 (ii)]{MR5066465}.
\end{proof}
\begin{proposition}\label{prop:relative-picard-finiteness}
Let $k$ be an algebraically closed field, and let $\pi: \cX \to X$ be an
adequate moduli space, where $\cX$ is an algebraic stack of finite type
over $k$ with affine diagonal. Then the quotient $\Lambda_\pi:=\Pic(\cX/X)/\Pic(\cX/X)_{\tors}$ is a free abelian group of finite rank. In particular, we have $\rho_\pi < \infty$.
\end{proposition}

\begin{proof}
Choose an atlas $U \twoheadrightarrow \cX$ with $U$ a scheme of finite type over $k$. By generic flatness and noetherian induction, $U_{\red}$
admits a finite stratification $U_{\red}=\sqcup_{i} U_i$ by integral locally closed subschemes
$U_i$ such that $H_i:=\cI_{\cX} \times_{\cX} U_i \to U_i$ is flat, where $\cI_{\cX}$ is the inertia stack of $\cX$. Let $\eta_i$ be a geometric generic point of $U_i$, and let $G_i:=(H_i)_{\eta_i}$. Stabilizer actions on the fibers of line bundles induce a homomorphism
\[
w: \Pic(\cX/X) \to \bigoplus_i X^*(G_i)/X^*(G_i)_{\tors}.
\]
The target is free of finite rank, since each $X^*(G_i)$ is finitely
generated \cite[Lem. 4.9]{MR3503967}. To conclude, we show that $\ker w = \Pic(\cX/X)_{\tors}$, and hence $\Lambda_\pi$ embeds into a free abelian group of finite rank and is therefore itself is free of finite rank.

Suppose that $w([\cL])=0$.  We show $\cL$ is torsion. By definition, there exists an integer $m>0$ such that the character of stabilizer at $\cL^{\otimes m}$ is trivial at every generic point $\eta_i$. Since $H_i$ is flat over the integral scheme $U_i$, the corresponding character $H_i \to \mathbb{G}_{m,U_i}$ is trivial everywhere. Thus every geometric stabilizer of $\cX$ acts trivially on $\cL^{\otimes m}$. By \Cref{lem:adequate-power-descent}, some positive power of $\cL^{\otimes m}$ descends to $X$, so $\cL$ itself is torsion.
\end{proof}
Now we can prove the main result of this section.
\begin{theorem}\label{thm:stack-picard-bound}
Let $k$ be an algebraically closed field, let $\cX$ be a smooth locally reductive integral algebraic stack of finite type over $k$ with affine diagonal, and let $\pi: \cX \to X$ be a stable adequate moduli space. Then
\[
|\PD(\pi)| \leq 2 \rho_\pi \text{ and } |\Op(\pi)| \leq 3^{\rho_\pi}.
\]
Both upper bounds are sharp.
\end{theorem}

\begin{proof}
By \cite[Thm. 5.3.1 and Prop. 5.4.1]{AlperAdequate}, the morphism $\pi$ is surjective and universally closed, and $X$ is integral. In particular, the pullback $\pi^*: k(X) \to k(\cX)$ is injective. Since $\pi$ is also stable, actually we have
\begin{claim}\label{claim:function-field-equal}
The pullback $\pi^*: k(X) \xrightarrow{\sim} k(\cX)$ is an isomorphism.
\end{claim}
\begin{proof}[Proof of Claim \ref{claim:function-field-equal}]
Let $f \in k(\cX)$, and choose an open dense substack $\cU \subseteq \cX$ on which $f$ is regular. Since stable points are dense, there is a stable point $x \in |\cU|$. The fiber over $\pi(x)$ consists only of $x$ itself, so $U:=X \setminus \pi(\cX \setminus \cU) \subseteq X$ is a non-empty open subset of $X$ with $\pi^{-1}(U) \subseteq \cU$.
Then $\cO_X \simeq \pi_*\cO_{\cX}$ gives $\Gamma(U,\cO_X) \xrightarrow{\sim} \Gamma(\pi^{-1}(U),\cO_{\cX})$. This shows that $f|_{\pi^{-1}(U)}$ is pulled back from a regular function on $U$, proving that $f$ lies in the image of $k(X)$.
\end{proof}
Suppose $\PD(\pi)=\{\cD_1,\dots,\cD_{r}\}$. Consider the morphism
\[
\alpha: \bbQ^{r} \to \Pic(\cX/X) \otimes_{\bbZ} \bbQ \text{ sending } e_i \mapsto [\cO_{\cX}(\cD_i)].
\]
where $e_i$ is the $i$-th standard basis vector. For any $a=(a_i) \in \ker \alpha$, by definition there exists an integer $n>0$ such that $\cO_{\cX}\left(n \sum_{i=1}^{r} a_i \cD_i\right) \simeq \pi^*L$ for some line bundle $L$ on $X$.
\begin{claim}\label{claim:unique-B-on-X}
There exists a unique $\bbQ$-Cartier divisor $D$ on $X$ such that $\sum_{i=1}^{r} a_i\cD_i=\pi^*D$.
\end{claim}
\begin{proof}[Proof of Claim \ref{claim:unique-B-on-X}]
Let $\cD:=\sum_{i=1}^{r} a_i \cD_i$, and fix an isomorphism $\varphi: \cO_{\cX}(n\cD) \xrightarrow{\sim} \pi^*L$. Let $\sigma_{n\cD}$ be the canonical rational section of $\cO_{\cX}(n\cD)$, so that $\ddiv(\sigma_{n\cD})=n\cD$, and choose a non-zero
rational section $s$ of $L$. Consider $g:=\varphi(\sigma_{n\cD})/\pi^*s \in k(\cX)^\times$. Since $\pi^*:k(X) \xrightarrow{\sim} k(\cX)$ is an isomorphism by Claim \ref{claim:function-field-equal}, there is a unique $f \in k(X)^\times$ such that $g=\pi^*f$. Hence $\varphi(\sigma_{n\cD})=\pi^*(fs)$. Taking divisors of rational sections gives $n\cD =\ddiv \bigl(\varphi(\sigma_{n\cD})\bigr)=\pi^*\ddiv(fs)$. Define $D:=\ddiv(fs)/n=\bigl(\ddiv(s)+\ddiv(f)\bigr)/n$, and it is a $\bbQ$-Cartier divisor satisfying $\cD=\pi^*D$.
\end{proof}
Note that for any prime divisor $E$ on $X$, the closure of every irreducible component of $\pi^{-1}(\eta_E)$ in $\cX$ has codimension one. Let $D_1,\dots,D_q$ be the prime divisors on $X$ such that the closure of each irreducible component of $\pi^{-1}(\eta_{D_j})$ in $\cX$, equipped with its reduced structure, is one of the $\cD_i$. For $j=1,\dots,q$, set
\[
I_j:=\{i \in \{1,\dots,r\}: \pi(\cD_i)=D_j\}.
\]
The sets $I_j$ are pairwise disjoint and $|I_j| \geq 2$. Indeed, if $I_j=\{i\}$, then $\cX \setminus \cD_i$ would have empty fiber over $\eta_{D_j}$, contradicting the surjectivity required by removability. Note that this gives a subset $\sqcup_{j=1}^q I_j \subseteq \PD(\pi)$ and hence
\[
r=|\PD(\pi)| \geq \sum_{j=1}^q |I_j| \geq \sum_{j=1}^q 2 = 2q.
\]

For the $\bbQ$-Cartier divisor $D$ on $X$ in Claim \ref{claim:unique-B-on-X}, every component of $\Supp(D)$ is one of the $D_j$. Since $X$ is normal \cite[Prop. 5.4.1]{AlperAdequate}, the divisor $D$ is uniquely determined by its coefficients along $D_1,\dots,D_q$. Therefore
\[
\dim_{\bbQ} \ker \alpha \leq q \text{ and hence } r-q \leq \rho_\pi.
\]
Combining this with $r \ge 2q$, we obtain $r \leq 2 \rho_\pi$. 

By \Cref{cor:good-finite}, every preserving open corresponds to a unique subset of $\PD(\pi)$, and a subset of $\PD(\pi)$ defining a preserving open cannot contain all of $I_j$, as the complement would miss the fiber over $\eta_{D_j}$. This observation gives an upper bound
\[
|\Op(\pi)| \leq 2^{r-\sum_{j=1}^q |I_j|} \prod_{j=1}^q \left(2^{|I_j|}-1\right) \leq 3^{r-\sum_{j=1}^q |I_j|} \prod_{j=1}^q \left(3^{|I_j|-1}\right) = 3^{r-q} \leq 3^{\rho_\pi},
\]
where we use $2^k-1 \leq 3^{k-1}$ for $k \geq 2$. Sharpness is achieved by \Cref{ex:axes} as $\Pic([\bbA^{2r}/\bbT]) \simeq X^*(\bbT) \simeq \bbZ^r$ and $\Pic(\bbA^r)=0$, so we have $\rho_\pi=r$.
\end{proof}
\subsection{The case of quotient stacks}
Let $G$ be a reductive group, and let $X$ be a smooth integral variety with a $G$-action that admits a good quotient $q: X \to Y$. We can write $q=\pi \circ p$, where $\pi: [X/G] \to Y$ is the induced adequate moduli space and $p: X \to [X/G]$ is the canonical atlas. Passing to the quotient stack identifies $\PD(q)$ with $\PD(\pi)$ and $\Op(q)$ with $\Op(\pi)$. Under the identification $\Pic([X/G])=\Pic^G(X)$, we have
\begin{gather*}
\rho_\pi=\rk \Pic([X/G]/Y) = \rk \Pic^G(X)/\pi^*\Pic(Y),  \\ 
\rho_q=\rk \Pic(X/Y) = \rk \Pic(X)/q^*\Pic(Y).
\end{gather*}
This suggests relating $\rho_\pi$ to $\rho_q$ by comparing the equivariant and ordinary Picard groups. Indeed, these two are related by an exact sequence.
\begin{lemma}\label{prop:reductive-picard-comparison}
There is an exact sequence
\[
X^*(G)_{\bbQ} \xrightarrow{\beta} (\Pic^G(X)/\pi^*\Pic(Y))_{\bbQ} \xrightarrow{\rF} (\Pic(X)/q^*\Pic(Y))_{\bbQ} \to 0,
\]
where $\beta(\chi)=\cO_X(\chi)$, with $\cO_X(\chi)$ the trivial line bundle with the $G$-linearization given by $\chi$, and $\rF$ is the forgetful map. If $X^G \neq \varnothing$, this extends to a split short exact sequence
\[
0 \to X^*(G)_{\bbQ} \xrightarrow{\beta} (\Pic^G(X)/\pi^*\Pic(Y))_{\bbQ}
\xrightarrow{\rF} (\Pic(X)/q^*\Pic(Y))_{\bbQ} \to 0.
\]
\end{lemma}
\begin{proof}
By \cite[Prop. 2.10 and Thm. 2.14]{Brion}, there is an exact sequence
\[
X^*(G)_{\bbQ} \to \Pic^G(X)_{\bbQ} \xrightarrow{\rF} \Pic(X)_{\bbQ} \to 0.
\]
The forgetful map $\rF$ carries $\pi^*\Pic(Y)_{\bbQ}$ onto $q^*\Pic(Y)_{\bbQ}$, so passing to the quotients gives the first sequence. If $x \in X^G(k)$, taking the character of the $G$-action
on the fiber at $x$ defines a morphism $\wt_x: \Pic^G(X) \to X^*(G)$. It vanishes on $\pi^*\Pic(Y)$ and sends $\cO_X(\chi)$ to $\chi$. Passing to the quotients and tensoring with $\bbQ$, it therefore gives a section of $\beta$, proving the second assertion.
\end{proof}
\begin{corollary}\label{cor:reductive-relative-picard-comparison}
We have $\rho_q+\rk G/[G,G] \geq \rho_\pi \geq \rho_q$. If in addition $X^G \neq \varnothing$, then $\rho_\pi=\rho_q+\rk G/[G,G]$.
\end{corollary}
\begin{proof}
Take ranks in \Cref{prop:reductive-picard-comparison}, using $\rk X^*(G)=\rk G/[G,G]$.
\end{proof}
\begin{corollary}\label{cor:ordinary-picard-counting}
Assume that $X^s \neq \varnothing$. Then
\[
|\PD(q)| \leq 2(\rho_q+\rk G/[G,G]) \text{ and } |\Op(q)| \leq 3^{\rho_q+\rk G/[G,G]}.
\]
If in addition $\Pic(X)$ is torsion, then
\[
|\PD(q)| \leq 2\rk G/[G,G] \text{ and } |\Op(q)| \leq 3^{\rk G/[G,G]}.
\]
\end{corollary}
\begin{proof}
Follows from \Cref{thm:stack-picard-bound} and \Cref{cor:reductive-relative-picard-comparison}. If $\Pic(X)$ is torsion, then $\rho_q=0$.
\end{proof}
\begin{example}\label{ex:torus-no-stable-unbounded}
The assumption $X^s \neq \varnothing$ in \Cref{cor:ordinary-picard-counting} cannot be omitted in general, even for torus actions on affine space. Indeed, let $n \geq 3$, and let $\bbT=\mathbb{G}_m^3$ act on $X=\bbA^n$ by
\[
(s,t,u).(x_j)_{j=0}^{n-1}:=\bigl(st^ju^{j^2}x_j\bigr)_{j=0}^{n-1}.
\]
The origin is the unique closed orbit, and $X^s=\varnothing$, and the good quotient is $q: X \to \Spec k$. Since $\Pic(X)=0$ and the origin is fixed, \Cref{cor:reductive-relative-picard-comparison} gives $\rho_q=0$ and $\rho_\pi=3$.

For each $i$, let $U_i:=\mathbf{D}(x_i) \subseteq X$ and consider the cocharacter $\lambda_i(z)=(z^{i^2},z^{-2i},z)$. The $\lambda_i$-weight on $x_j$ is $(j-i)^2$, so $k[U_i]^{\lambda_i}=k[x_i^{\pm 1}]$. Since the first factor of $\bbT$ scales $x_i$, it follows that
$k[U_i]^\bbT=k$. Thus each affine open $U_i \subseteq X$ preserves $q$,
and all $n$ coordinate hyperplanes are removable. Therefore
\[
|\PD(q)|\geq n \text{ and } |\Op(q)|\geq n+1.
\]
Neither is bounded by the relative Picard rank.
\end{example}
For $\Gm$-actions, however, an empty stable locus gives even better upper bounds.
\begin{lemma}\label{lem:generalgm}
Assume $G=\Gm$. If $X^s=\varnothing$, then $|\PD(q)| \leq 1 \leq 2\rho_\pi$ and $|\Op(q)| \leq 2 \leq 3^{\rho_\pi}$.
\end{lemma}
\begin{proof}
Since $X^s=\varnothing$, we have $X=\cup_i X_i^\pm$, and each $X_i^\pm$ is closed in $X$. Irreducibility therefore gives $X=X_i^+$ or $X=X_i^-$ for some $i$. In particular, the fixed locus $X^{\Gm}$ is non-empty and connected. By \Cref{thm:main-Gm}, the only possible non-trivial open preserving $q$ is $X \setminus X^{\Gm}$. This proves the first two bounds. Finally, by \Cref{cor:reductive-relative-picard-comparison} we have $\rho_\pi=\rho_q+1\geq1$, proving the remaining inequalities.
\end{proof}

\subsection{\texorpdfstring{The fixed loci of $\Gm$-actions}{The fixed loci of Gm-actions}}\label{sec:fixed-locus}
For $\Gm$-actions, taking weights at each fixed component relate the relative Picard group to the connected components of the fixed locus. This simple observation will give a lower bound for the number of the fixed components.
 
Let $X$ be a smooth integral variety with a $\Gm$-action that admits a good quotient $q: X \to Y$. The fixed locus $X^\Gm$ is smooth and let $m:=|\pi_0(X^\Gm)|$. Note that line bundles that pulled back from $Y$ have weight $0$, so taking fiber weights defines a well-defined morphism
\[
w: \Pic([X/\Gm]/Y) \to \bbZ^m \text{ sending } L \mapsto \bigl(\wt_{\id}(L|_{X_i})\bigr)_{i \in \pi_0(X^\Gm)}.
\]

The same proof of \Cref{prop:relative-picard-finiteness} gives $\ker w =\Pic([X/\Gm]/Y)_{\tors}$ and hence
\begin{proposition}\label{prop:fixed-weight-kernel}
We have $\Lambda_\pi \hookrightarrow \bbZ^m$ and hence $m \geq \rho_\pi$.
\end{proposition}
The next result gives a geometric condition under which $m=\rho_\pi$.
\begin{theorem}\label{thm:fixed-components}
Assume that $\codim_X X_i \leq 3$ for each $i \in \pi_0(X^\Gm)$. Then $m=\rho_\pi$. If in addition $X^\Gm \neq \varnothing$, then $m=1+\rho_q$.
\end{theorem}

\begin{proof}
For $m \leq 1$, the equality follows from \Cref{prop:fixed-weight-kernel}. Suppose $m \geq 2$, and let
\[
p_i:=\codim_X X_i^- \text{ and } q_i:=\codim_X X_i^+.
\]
The loci $X_i^\pm$ are closed in $X$. Thus $p_i=0$ or $q_i=0$
would force the corresponding locus to equal $X$ and hence
$X_i=X^\Gm$, contradicting $m \geq 2$. Therefore
\[
p_i, q_i \geq 1 \text{ and } p_i+q_i=\codim_X X_i \leq 3.
\]
In particular, at least one of $X_i^\pm$ is a divisor. If $X_i^+$ is a divisor, then the weight of the line bundle $\cO_X(X_i^+)$ is $0$ on each $X_j$ with $j \neq i$ and non-zero on $X_i$, by \Cref{cor:weight-inclusion}. Hence $w(\cO_X(X_i^+))=c_i e_i$ for some $c_i \neq 0$. These $m$ vectors are linearly independent, so $m \geq \rho_\pi \geq m$, i.e., $m=\rho_\pi$. If $X^\Gm \neq \varnothing$, then $m=1+\rho_q$ follows from \Cref{cor:reductive-relative-picard-comparison}.
\end{proof}

\begin{corollary}\label{cor:fixed-torsion-picard}
Assume that $\Pic(X)$ is torsion (e.g., $X=\Spec A$ for some UFD $A$). Then $|\Op(q)| \leq 3$. If in addition $\dim X \leq 3$, then $X^\Gm$ is either empty or irreducible.
\end{corollary}

\begin{proof}
By \Cref{cor:reductive-relative-picard-comparison} we have $\rho_\pi \leq 1$, and hence $|\Op(q)| \leq 3$. If $X^\Gm \neq \varnothing$, then by \Cref{thm:fixed-components} we have $m=1+\rho_q=1$.
\end{proof}

\bibliographystyle{amsalpha}
\bibliography{stack_moduli_rigidity_references}

@article{AlperAdequate,
    AUTHOR = {Alper, Jarod},
     TITLE = {Adequate moduli spaces and geometrically reductive group
              schemes},
   JOURNAL = {Algebr. Geom.},
  FJOURNAL = {Algebraic Geometry},
    VOLUME = {1},
      YEAR = {2014},
    NUMBER = {4},
     PAGES = {489--531},
      ISSN = {2313-1691,2214-2584},
   MRCLASS = {14D20 (14A20 14L15 14L24 14L30)},
  MRNUMBER = {3272912},
MRREVIEWER = {Fabio\ Tonini},
       DOI = {10.14231/AG-2014-022},
       URL = {https://doi.org/10.14231/AG-2014-022},
}

@misc{fernandez_herrero_distinguishing_2024,
  author = {Fernandez Herrero, Andres and Wei{\ss}mann, Dario and Zhang, Xucheng},
  title = {Distinguishing algebraic spaces from schemes},
  year = {2024},
  eprint = {2411.07169},
  archivePrefix = {arXiv},
  note = {\href{https://arxiv.org/abs/2411.07169}{arXiv:2411.07169}}
}

@article {MR422314,
    AUTHOR = {de Concini, C. and Procesi, C.},
     TITLE = {A characteristic free approach to invariant theory},
   JOURNAL = {Advances in Math.},
  FJOURNAL = {Advances in Mathematics},
    VOLUME = {21},
      YEAR = {1976},
    NUMBER = {3},
     PAGES = {330--354},
      ISSN = {0001-8708},
   MRCLASS = {15A72 (14L99)},
  MRNUMBER = {422314},
MRREVIEWER = {Vladimir\ L.\ Popov},
       DOI = {10.1016/S0001-8708(76)80003-5},
       URL = {https://doi.org/10.1016/S0001-8708(76)80003-5},
}

@article {MR485944,
    AUTHOR = {D\'esarm\'enien, J. and Kung, Joseph P. S. and Rota,
              Gian-Carlo},
     TITLE = {Invariant theory, {Y}oung bitableaux, and combinatorics},
   JOURNAL = {Advances in Math.},
  FJOURNAL = {Advances in Mathematics},
    VOLUME = {27},
      YEAR = {1978},
    NUMBER = {1},
     PAGES = {63--92},
      ISSN = {0001-8708},
   MRCLASS = {15A72 (14M15)},
  MRNUMBER = {485944},
       DOI = {10.1016/0001-8708(78)90077-4},
       URL = {https://doi.org/10.1016/0001-8708(78)90077-4},
}

@article {MR979587,
    AUTHOR = {Richman, David R.},
     TITLE = {The fundamental theorems of vector invariants},
   JOURNAL = {Adv. in Math.},
  FJOURNAL = {Advances in Mathematics},
    VOLUME = {73},
      YEAR = {1989},
    NUMBER = {1},
     PAGES = {43--78},
      ISSN = {0001-8708},
   MRCLASS = {20C07 (05A99 15A72)},
  MRNUMBER = {979587},
MRREVIEWER = {Andrea\ Brini},
       DOI = {10.1016/0001-8708(89)90059-5},
       URL = {https://doi.org/10.1016/0001-8708(89)90059-5},
}

@article {Luna-slice-for-stacks,
    AUTHOR = {Alper, Jarod and Hall, Jack and Rydh, David},
     TITLE = {A {L}una \'etale slice theorem for algebraic stacks},
   JOURNAL = {Ann. of Math. (2)},
  FJOURNAL = {Annals of Mathematics. Second Series},
    VOLUME = {191},
      YEAR = {2020},
    NUMBER = {3},
     PAGES = {675--738},
      ISSN = {0003-486X,1939-8980},
   MRCLASS = {14D23 (14A20 14F20 14L24 14L30)},
  MRNUMBER = {4088350},
MRREVIEWER = {Martin\ Ulirsch},
       DOI = {10.4007/annals.2020.191.3.1},
       URL = {https://doi.org/10.4007/annals.2020.191.3.1},
}

@article {Brion,
    AUTHOR = {Brion, Michel},
     TITLE = {On linearization of line bundles},
   JOURNAL = {J. Math. Sci. Univ. Tokyo},
  FJOURNAL = {The University of Tokyo. Journal of Mathematical Sciences},
    VOLUME = {22},
      YEAR = {2015},
    NUMBER = {1},
     PAGES = {113--147},
      ISSN = {1340-5705},
   MRCLASS = {14C20 (20G15)},
  MRNUMBER = {3329192},
MRREVIEWER = {Halszka\ Tutaj-Gasi\'nska},
}

@misc{stacks-project,
  author       = {The {Stacks project authors}},
  title        = {The Stacks project},
  howpublished = {\url{https://stacks.math.columbia.edu}},
  year         = {2026},
}

@article {second-flip,
    AUTHOR = {Alper, Jarod and Fedorchuk, Maksym and Smyth, David Ishii and
              van der Wyck, Frederick},
     TITLE = {Second flip in the {H}assett-{K}eel program: a local
              description},
   JOURNAL = {Compos. Math.},
  FJOURNAL = {Compositio Mathematica},
    VOLUME = {153},
      YEAR = {2017},
    NUMBER = {8},
     PAGES = {1547--1583},
      ISSN = {0010-437X,1570-5846},
   MRCLASS = {14H10 (14D23 14E30 14L30)},
  MRNUMBER = {3705268},
MRREVIEWER = {Scott\ R.\ Nollet},
       DOI = {10.1112/S0010437X16008290},
       URL = {https://doi.org/10.1112/S0010437X16008290},
}

@article{dario-zhang,
    AUTHOR = {Wei{\ss}mann, Dario and Zhang, Xucheng},
     TITLE = {A stacky approach to identifying the semistable locus of
              bundles},
   JOURNAL = {Algebr. Geom.},
  FJOURNAL = {Algebraic Geometry},
    VOLUME = {12},
      YEAR = {2025},
    NUMBER = {2},
     PAGES = {262--298},
      ISSN = {2313-1691,2214-2584},
   MRCLASS = {14H60 (14D20 14D23)},
  MRNUMBER = {4869967},
MRREVIEWER = {Francesco\ Bottacin},
       DOI = {10.14231/ag-2025-008},
       URL = {https://doi.org/10.14231/ag-2025-008},
}

@article {AHLH,
    AUTHOR = {Alper, Jarod and Halpern-Leistner, Daniel and Heinloth,
              Jochen},
     TITLE = {Existence of moduli spaces for algebraic stacks},
   JOURNAL = {Invent. Math.},
  FJOURNAL = {Inventiones Mathematicae},
    VOLUME = {234},
      YEAR = {2023},
    NUMBER = {3},
     PAGES = {949--1038},
      ISSN = {0020-9910,1432-1297},
   MRCLASS = {14D23 (14A20)},
  MRNUMBER = {4665776},
MRREVIEWER = {LingGuang\ Li},
       DOI = {10.1007/s00222-023-01214-4},
       URL = {https://doi.org/10.1007/s00222-023-01214-4},
}

@article {AlperGood,
    AUTHOR = {Alper, Jarod},
     TITLE = {Good moduli spaces for {A}rtin stacks},
   JOURNAL = {Ann. Inst. Fourier (Grenoble)},
  FJOURNAL = {Universit\'e{} de Grenoble. Annales de l'Institut Fourier},
    VOLUME = {63},
      YEAR = {2013},
    NUMBER = {6},
     PAGES = {2349--2402},
      ISSN = {0373-0956,1777-5310},
   MRCLASS = {14D23 (14L24 14L30)},
  MRNUMBER = {3237451},
MRREVIEWER = {Arvid\ Siqveland},
       DOI = {10.5802/aif.2833},
       URL = {https://doi.org/10.5802/aif.2833},
}

@article {ADL,
    AUTHOR = {Ascher, Kenneth and DeVleming, Kristin and Liu, Yuchen},
     TITLE = {Wall crossing for {K}-moduli spaces of plane curves},
   JOURNAL = {Proc. Lond. Math. Soc. (3)},
  FJOURNAL = {Proceedings of the London Mathematical Society. Third Series},
    VOLUME = {128},
      YEAR = {2024},
    NUMBER = {6},
     PAGES = {Paper No. e12615, 113},
      ISSN = {0024-6115,1460-244X},
   MRCLASS = {14H10 (14H50 14J10 14J45 14L24)},
  MRNUMBER = {4757702},
MRREVIEWER = {Luca\ Schaffler},
       DOI = {10.1112/plms.12615},
       URL = {https://doi.org/10.1112/plms.12615},
}

@misc{Drinfeld,
  author = {Drinfeld, Vladimir},
  title = {On algebraic spaces with an action of {$\mathbb G_m$}},
  year = {2013},
  eprint = {1308.2604},
  archivePrefix = {arXiv},
  note = {\href{https://arxiv.org/abs/1308.2604}{arXiv:1308.2604}}
}

@article {MR5066465,
    AUTHOR = {Belmans, Pieter and Damiolini, Chiara and Franzen, Hans and
              Hoskins, Victoria and Makarova, Svetlana and Tajakka, Tuomas},
     TITLE = {Projectivity and effective global generation of determinantal
              line bundles on quiver moduli},
   JOURNAL = {Algebra Number Theory},
  FJOURNAL = {Algebra \& Number Theory},
    VOLUME = {20},
      YEAR = {2026},
    NUMBER = {4},
     PAGES = {747--800},
      ISSN = {1937-0652,1944-7833},
   MRCLASS = {14D23 (16G20)},
  MRNUMBER = {5066465},
MRREVIEWER = {Bidhan\ Paul},
       DOI = {10.2140/ant.2026.20.747},
       URL = {https://doi.org/10.2140/ant.2026.20.747},
}

@article {MR3503967,
    AUTHOR = {Hashimoto, Mitsuyasu},
     TITLE = {Equivariant class group. {I}. {F}inite generation of the
              {P}icard and the class groups of an invariant subring},
   JOURNAL = {J. Algebra},
  FJOURNAL = {Journal of Algebra},
    VOLUME = {459},
      YEAR = {2016},
     PAGES = {76--108},
      ISSN = {0021-8693,1090-266X},
   MRCLASS = {13A50 (13C20 14L15)},
  MRNUMBER = {3503967},
MRREVIEWER = {Kiyoshi\ Baba},
       DOI = {10.1016/j.jalgebra.2016.02.025},
       URL = {https://doi.org/10.1016/j.jalgebra.2016.02.025},
}

@article {MR4255045,
    AUTHOR = {Edidin, Dan and Rydh, David},
     TITLE = {Canonical reduction of stabilizers for {A}rtin stacks with
              good moduli spaces},
   JOURNAL = {Duke Math. J.},
  FJOURNAL = {Duke Mathematical Journal},
    VOLUME = {170},
      YEAR = {2021},
    NUMBER = {5},
     PAGES = {827--880},
      ISSN = {0012-7094,1547-7398},
   MRCLASS = {14D23 (14E15 14L24)},
  MRNUMBER = {4255045},
MRREVIEWER = {Ariyan\ Javanpeykar},
       DOI = {10.1215/00127094-2020-0050},
       URL = {https://doi.org/10.1215/00127094-2020-0050},
}

@article {Benoist2013,
    AUTHOR = {Benoist, Olivier},
     TITLE = {Quasi-projectivity of normal varieties},
   JOURNAL = {Int. Math. Res. Not. IMRN},
  FJOURNAL = {International Mathematics Research Notices. IMRN},
      YEAR = {2013},
    NUMBER = {17},
     PAGES = {3878--3885},
      ISSN = {1073-7928,1687-0247},
   MRCLASS = {14C22 (14C20 14L30)},
  MRNUMBER = {3096912},
MRREVIEWER = {Liam\ O'Carroll},
       DOI = {10.1093/imrn/rns163},
       URL = {https://doi.org/10.1093/imrn/rns163},
}

@misc{halpern-leistner2014instability,
  author       = {Halpern-Leistner, Daniel},
  title        = {On the structure of instability in moduli theory},
  year         = {2014},
  eprint       = {1411.0627},
  archivePrefix = {arXiv},
  primaryClass = {math.AG},
  doi          = {10.48550/arXiv.1411.0627},
  url          = {https://arxiv.org/abs/1411.0627},
  note = {\href{https://arxiv.org/abs/1411.0627}{arXiv:1411.0627}}
}

@article{Heinloth,
    AUTHOR = {Heinloth, Jochen},
     TITLE = {Hilbert-{M}umford stability on algebraic stacks and applications to {$\mathcal G$}-bundles on curves},
   JOURNAL = {\'Epijournal G\'eom. Alg\'ebrique},
  FJOURNAL = {\'Epijournal de G\'eom\'etrie Alg\'ebrique. EPIGA},
    VOLUME = {1},
      YEAR = {2017},
     PAGES = {Art. 11, 37},
      ISSN = {2491-6765},
   MRCLASS = {14L24 (14A20 14D23)},
  MRNUMBER = {3758902},
MRREVIEWER = {Scott\ R.\ Nollet},
       DOI = {10.46298/epiga.2018.volume1.2062},
       URL = {https://doi.org/10.46298/epiga.2018.volume1.2062},
}

@article {ADLQuarticK3,
    AUTHOR = {Ascher, Kenneth and DeVleming, Kristin and Liu, Yuchen},
     TITLE = {K-stability and birational models of moduli of quartic {K}3
              surfaces},
   JOURNAL = {Invent. Math.},
  FJOURNAL = {Inventiones Mathematicae},
    VOLUME = {232},
      YEAR = {2023},
    NUMBER = {2},
     PAGES = {471--552},
      ISSN = {0020-9910,1432-1297},
   MRCLASS = {14J28 (14J10 14J45)},
  MRNUMBER = {4574660},
MRREVIEWER = {Guolei\ Zhong},
       DOI = {10.1007/s00222-022-01170-5},
       URL = {https://doi.org/10.1007/s00222-022-01170-5},
}

@misc{ABBDeVILW23,
  author        = {Ascher, Kenneth and Bejleri, Dori and Blum, Harold and DeVleming, Kristin and Inchiostro, Giovanni and Liu, Yuchen and Wang, Xiaowei},
  title         = {Moduli of boundary polarized {Calabi--Yau} pairs},
  year          = {2023},
  eprint        = {2307.06522},
  archivePrefix = {arXiv},
  note = {\href{https://arxiv.org/abs/2307.06522}{arXiv:2307.06522}}
}

@article {Symonds-Gm,
    AUTHOR = {Symonds, Peter},
     TITLE = {Smith theory for algebraic varieties},
   JOURNAL = {Algebr. Geom. Topol.},
  FJOURNAL = {Algebraic \& Geometric Topology},
    VOLUME = {4},
      YEAR = {2004},
     PAGES = {121--131},
      ISSN = {1472-2747,1472-2739},
   MRCLASS = {57S17 (14F20 55M20 55M35)},
  MRNUMBER = {2059185},
MRREVIEWER = {Claudio\ Pedrini},
       DOI = {10.2140/agt.2004.4.121},
       URL = {https://doi.org/10.2140/agt.2004.4.121},
}

@misc{BHIK2025,
  author = {Bu, Chenjing and Halpern-Leistner, Daniel and
            Ib{\'a}{\~n}ez N{\'u}{\~n}ez, Andr{\'e}s and Kinjo, Tasuki},
  title = {Intrinsic {Donaldson--Thomas} theory. {I}. {C}omponent lattices of stacks},
  year = {2025},
  eprint = {2502.13892},
  archivePrefix = {arXiv},
  note = {\href{https://arxiv.org/abs/2502.13892}{arXiv:2502.13892}}
}

@article {DolgachevHu1998,
    AUTHOR = {Dolgachev, Igor V. and Hu, Yi},
     TITLE = {Variation of geometric invariant theory quotients},
      NOTE = {With an appendix by Nicolas Ressayre},
   JOURNAL = {Inst. Hautes \'Etudes Sci. Publ. Math.},
  FJOURNAL = {Institut des Hautes \'Etudes Scientifiques. Publications
              Math\'ematiques},
    NUMBER = {87},
      YEAR = {1998},
     PAGES = {5--56},
      ISSN = {0073-8301,1618-1913},
   MRCLASS = {14L24},
  MRNUMBER = {1659282},
MRREVIEWER = {P.\ E.\ Newstead},
       URL = {http://www.numdam.org/item?id=PMIHES_1998__87__5_0},
}

@misc{DLIStableMaps,
  author        = {Di Lorenzo, Andrea and Inchiostro, Giovanni},
  title         = {Stable maps to quotient stacks with a properly stable point},
  year          = {2024},
  eprint        = {2411.16141},
  archivePrefix = {arXiv},
  note = {\href{https://arxiv.org/abs/2411.16141}{arXiv:2411.16141}}
}
\end{document}